\documentclass[12pt]{amsart}

\usepackage[margin=30truemm]{geometry}
\usepackage{amssymb} 
\usepackage{braket}
\usepackage{bm}
\usepackage{comment}
\usepackage{mathtools}
\usepackage{color}

\numberwithin{equation}{section}

\theoremstyle{plain} 
\newtheorem{theorem}{Theorem}[section]
\newtheorem{lemma}[theorem]{Lemma}
\newtheorem{corollary}[theorem]{Corollary}
\newtheorem{proposition}[theorem]{Proposition}
\newtheorem{claim}[theorem]{Claim}
\newtheorem{remark}[theorem]{Remark}

\theoremstyle{definition}
\newtheorem{definition}[theorem]{Definition}

\newtheorem*{ackn}{\bf{Acknowledgment}}
\newtheorem*{orga}{\bf{Organization}}

\newtheorem*{ai}{Disclosure on the use of AI}

\newcommand{\Ker}[0]{\operatorname{Ker}}

\newcommand{\Aut}{\mathrm{Aut}}

\newcommand{\Aff}{\mathrm{Aff}}
\newcommand{\deldelb}{\sqrt{-1}\partial_{B} \overline{\partial_{B}}}
\newcommand{\deldel}{\sqrt{-1}\partial \overline{\partial}}

\newcommand{\dbar}{\overline{\partial}}
\newcommand{\e}{\varepsilon}

\newcommand{\tr}{\mathrm{tr}}
\newcommand{\Ric}{\mathrm{Ric}}
\newcommand{\MA}{\mathrm{MA}}

\renewcommand{\a}{\alpha}
\renewcommand{\b}{\beta}

\renewcommand{\phi}{\varphi}

\renewcommand{\leq}{\leqslant}
\renewcommand{\geq}{\geqslant}

\newcommand{\ip}[1]{\left\langle#1\right\rangle}

\begin{document}

\title[Soliton-type metrics associated with weighted CSCK metrics]
{Soliton-type metrics associated with
weighted constant scalar curvature K\"ahler metrics
on Fano manifolds}
\author[S.~Nakamura]
{Satoshi Nakamura} 

\thanks{ 
}
\address{
}
\email{}

\address{
S. Nakamura:
Department of Mathematics, 
Nagoya Institute of Technology, 
Nagoya 466-8555, Japan
}
\email{nakamura.satoshi@nitech.ac.jp}



\maketitle

\begin{abstract}

We establish an existence equivalence between 
a weighted constant scalar curvature K\"ahler metric introduced by Lahdili
and an associated soliton-type metric introduced in this article on Fano manifolds. 
We apply it to semisimple principal fibrations with Fano fiber. 
This yields a family of extremal K\"ahler metrics,
including a constant scalar curvature K\"ahler metric,
on a non-Fano complex four-fold with non-toric Fano fiber. 
A further application concerns Sasaki metrics on canonical circle bundles over $n$-dimensional Fano manifolds. 
For a fixed (possibly irregular) Reeb field, we show that
there exists a Sasaki Mabuchi soliton
if and only if there exists a Sasaki extremal metric
whose transverse scalar curvature
is strictly less than $n+1$. 
Finally, we characterize Reeb fields on the standard sphere $S^{2n+1}$
that admit Sasaki Mabuchi solitons.

\end{abstract}

\tableofcontents

\section{Introduction}
In his paper \cite{Lah19}, Lahdili
introduced the notion of a weighted
constant scalar curvature K\"ahler metric (a weighted  CSCK metric for short)
formulated as follows.
Let $X$ be an $n$-dimensional compact K\"ahler manifold with a
K\"ahler class $\Omega$, and 
let $T$ be a maximal compact torus in
the reduced automorphism group $\mathrm{Aut}_{r}(X)$.
Fix two smooth functions $v>0$ and $w$ defined on
the moment map image $P_{X}$.
A $T$-invariant K\"ahler metric $\omega\in\Omega$ is called
 a $(v,w)$-constant scalar curvature K\"ahler metric
(a $(v,w)$-CSCK metric for short) if $S_{v}(\omega)=w(m_{\omega})$,
where $S_{v}(\omega)$ is the $v$-weighted scalar curvature of $\omega$ and 
$m_{\omega}$ is the moment map for the $T$-action
with respect to $\omega$.
When $v\equiv 1$, the weighted scalar curvature equals the ordinary scalar curvature.
The $(v,w)$-CSCK metrics provide a unifying framework that encompasses several
well-known classes of canonical metrics,
including
constant scalar curvature K\"ahler metrics, extremal K\"ahler metrics, 
Sasaki extremal metrics.
In recent years, $(v,w)$-CSCK metrics have attracted considerable attention. 
The existence of $(v,w)$-CSCK metrics with log-concave weight $v$
was shown by Han-Liu \cite{HanLiu25} and Di~Nezza-Jubert-Lahdili \cite{DJL24, DJL25}
to be equivalent to the coercivity of the associated Mabuchi-type energy functional. 
Moreover a Yau-Tian-Donaldson-type correspondence was established
by Boucksom-Jonsson \cite{BJ25} and Trusiani \cite{Tru26},
building on earlier work of Li \cite{CL22}
(see also \cite{HanLiu25b, Ha25}),
characterizing the existence of $(v,w)$-CSCK metrics with log-concave weight $v$
in terms of an appropriate notion of K-stability.
Very recently, the assumption on the log-concavity of $v$
was removed by Jubert-Pan \cite{JP26}.

In the case where $X$ is a Fano manifold and $\Omega=2\pi c_{1}(X)$,
soliton-type metrics called $g$-solitons,
where $g$ is a smooth positive function defined on $P_{X}$,
play a central role within the class of weighted CSCK metrics.
A $T$-invariant metric $\omega\in 2\pi c_{1}(X)$ is called a $g$-soliton
if $\Ric_{g}(\omega)=\omega$, where $\Ric_{g}(\omega)$ is the $g$-weighted
Ricci curvature form of $\omega$.
When $g\equiv 1$, the wighted Ricci form equals the ordinary Ricci form.
A $g$-soliton equals a $(g,w)$-CSCK metric with
$w(x)=n+\langle (\log g)^{\prime}(x),x\rangle$ for $x\in P_{X}$
as reviewed in Proposition \ref{g-sol-csck}.
The notion of $g$-solitons provides a unified framework encompassing
K\"ahler-Einstein metrics, K\"ahler-Ricci solitons, Mabuchi solitons, Sasaki-Einstein metrics.
Originating in Mabuchi's work \cite{Ma03} and further developed 
by Berman-Witt Nysr\"om \cite{BW14} and Han-Li \cite{HaLi20}
(together with Jubert-Pan \cite{JP26}),
the existence of $g$-solitons 
is by now well understood to be characterized by
the coercivity of the associated Ding-type functional
as well as by an appropriate Ding-type stability condition.

Recently, there has been growing progress in the study of relationships
between the existence of two, in general different, canonical K\"ahler metrics
on a Fano manifold.

\begin{theorem}{\rm (\cite{ALN25, HN24}, see also \cite[Section 9.6]{Ma21})}\label{meta-theorem}
Let $X$ be an $n$-dimensional Fano manifold.
Then the following conditions are equivalent.
\begin{enumerate}
\item There exists a $T$-invariant Calabi's extremal metric
in $2\pi c_{1}(X)$, whose scalar curvature is strictly less than $n+1$.
\item There exists a $T$-invariant Mabuchi soliton in $2\pi c_{1}(X)$.
\end{enumerate}
\end{theorem}

Here an extremal metric means a $(1,l_{\mathrm{ext}})$-CSCK metric
for an affine function $l_{\mathrm{ext}}$ on $P_{X}$
corresponds to the extremal vector field \cite{FM95} on $X$
(see Section \ref{Application:fibration} for more details),
and a Mabuchi soliton means a $(n+1-l_{\mathrm{ext}})$-soliton.
These metrics are distinct unless $l_{\rm ext}$ is constant \cite{Nit23}, 
or equivalently, unless the metrics are K\"ahler-Einstein.
Following \cite{Ma01}, define the Mabuchi constant as
\begin{equation}
M_{X}
:=
\sup_{P_X}
\bigl(l_{\mathrm{ext}}-n\bigr).
\end{equation}
The condition $M_{X}<1$
is equivalent to that the weight $n+1-l_{\rm ext}$ is strictly positive on $P_X$.
Theorem \ref{meta-theorem} shows that, whenever an extremal metric exists, 
the necessary condition $M_{X}<1$
is also sufficient for the existence of a Mabuchi soliton.

On a Fano manifold, the equivalence between K\"ahler-Einstein metrics
and CSCK metrics in the anticanonical class
follows from a simple trace argument.
Although this elementary fact is encompassed by Theorem \ref{meta-theorem},
the extremal metric equations and Mabuchi soliton equations
are not related by taking traces in general. 
The proof of Theorem \ref{meta-theorem} in \cite{HN24} 
relies on a comparison of the energy functionals associated with these two types of metrics and an analysis of the Monge-Amp\`ere equation
used to construct Mabuchi solitons. 
Thus, Theorem \ref{meta-theorem} establishes a nontrivial correspondence 
between the existence of two distinct types of canonical metrics.



Now we state the results in this article.
\subsection{An existence correspondence on Fano manifolds}
Let $v>0$ and $w$ be smooth functions on $P_X$ satisfying
\begin{equation}
\int_Xv(m_\omega)\omega^n=\int_X\omega^n
\quad\text{and}\quad
\int_{X}S_{v}(\omega)v(m_\omega)\omega^n
=\int_{X}w(m_{\omega})v(m_\omega)\omega^n.
\end{equation}
These integrals are independent of the choice
of a metric $\omega$ \cite{Lah19}.
As our first result, 
we formulate and prove a generalization of the existence correspondence 
established in Theorem \ref{meta-theorem} 
to the framework of $(v,w)$-CSCK metrics on Fano manifolds.


\begin{theorem}\label{main-theorem}
Let $X$ be an $n$-dimensional Fano manifold.
Assume the associated weight function
\begin{equation}
g(v,w)(x):=v(x)\left(1+n+\langle(\log v)^{\prime}(x),x\rangle -w(x)\right)
\end{equation}
is strictly positive 
on $P_{X}$.
Then the following conditions
are equivalent.
\begin{enumerate}
\item There exists a $T$-invariant $(v,w)$-CSCK metric in $2\pi c_{1}(X)$.
\item There exists a $T$-invariant $g(v,w)$-soliton in $2\pi c_{1}(X)$.
\end{enumerate}
\end{theorem}


In particular, Theorem \ref{main-theorem} shows that
the existence of a $(v,w)$-CSCK metric on a Fano manifold is governed by
the existence of a soliton-type metric,
namely the $g(v,w)$-soliton.

In view of \cite{ALN25, HN24} (see also \cite{NN25}), one obtains the following.

\begin{corollary}\label{sub-theorem}
Under the same assumptions as in Theorem \ref{main-theorem},
the following conditions are equivalent.
\begin{enumerate}
\item There exists a $T$-invariant $(v,w)$-CSCK metric in $2\pi c_{1}(X)$.
\item There exists a $T$-invariant $(1, 1+n-g(v,w))$-CSCK metric in $2\pi c_{1}(X)$.
\end{enumerate}
\end{corollary}

\begin{remark}
The author has been informed by V.~Apostolov and Y.~Nitta that
the same results as Theorem \ref{main-theorem} and Corollary \ref{sub-theorem} 
were obtained independently in a joint work with Apostolov-Lahdili-Nitta \cite{ALN25b}.
\end{remark}

The formulation of $g(v,w)$-solitons is based on the observation in Proposition \ref{FutMD}
that two Futaki-type invariants, $M_{v,w}^\prime$ and $D_{g(v,w)}^\prime$, coincide. 
The invariant $M_{v,w}^\prime$ provides an obstruction to the existence of $(v,w)$-cscK metrics,
and $D_{g(v,w)}^\prime$ provides an obstruction to the existence of $g(v,w)$-solitons.

Under the appropriate formulation of $g(v,w)$, 
the analytic framework developed by Hisamoto and the author in \cite{HN24}
admits an extension to the $(v,w)$-CSCK setting.
A crucial step in our proof of Theorem \ref{main-theorem}
is to establish Theorem \ref{bound-of-M} which shows
the uniform boundedness of the Mabuchi-type functional $M_{v,w}$
associated with $(v,w)$-CSCK metrics
along a continuity path used to construct a $g(v,w)$-soliton.


One of the main significances of Theorem \ref{main-theorem} is that the existence of 
$(v,w)$-CSCK metrics on Fano manifolds can be captured by the existence of 
$g(v,w)$-solitons, that is, by the solvability of a complex Monge-Amp\`ere equation.
The existence of $g$-solitons admits several characterizations:
by a weighted barycenter condition on $P_X$ in the toric Fano case \cite{BB13},
by a Ding-type stability condition \cite{HaLi20},
and by a weighted delta invariant \cite{DJ25}.
These results yield corresponding combinatorial
or stability-theoretic characterizations of the existence of $(v,w)$-CSCK metrics.
In particular, a Ding-type stability condition for $g$-soliton
is deeply connected with the theory of minimal model program
in birational geometry \cite{HaLi20}.
The author expects that Theorem \ref{main-theorem}
may be applied to the moduli theory of Fano manifolds admitting $(v,w)$-CSCK metrics.

\subsection{Application to extremal metrics}
The second result of this article is
an application of Theorem \ref{main-theorem}
to the new existence result of CSCK metrics and extremal metrics
on certain fibrations with Fano fiber. 
Let $Y \to B$ be a semisimple principal fibration \cite{AJL23,DJ23}, 
with fiber a Fano manifold $X$ equipped with
the action of a maximal compact torus $T$ in $\Aut(X)$,
and base a product of CSCK Hodge manifolds.
We consider compatible K\"ahler classes on $Y$
whose restriction to the fiber is $2\pi c_1(X)$.
These fibrations provide a natural setting in which Theorem \ref{main-theorem} 
can be used to study extremal metrics on a total space that need not be Fano.

Classical constructions provide extensive families of extremal K\"ahler metrics
on ruled manifolds and projective bundles, 
beginning with Calabi's examples \cite{Ca82} and extending to 
higher-dimensional bundle constructions \cite{ACGT11}. 
More generally, a result of Apostolov–Jubert–Lahdili \cite{AJL23} 
relates extremal metrics on semisimple principal fibrations to 
weighted CSCK metrics on the fiber.
When the fiber itself admits an extremal metric,
this framework also yields extremal metrics
on the total space of the fibration for {\it sufficiently large} base classes
\cite{ACGT04, DS21}. 
In the toric fiber setting, Delcroix–Jubert \cite{DJ23} 
obtained effective sufficient conditions
that can be checked directly from the moment polytope and the fibration data.

By \cite{AJL23}, 
the extremal metric equation for compatible K\"ahler metrics on $Y$
can be expressed as a $(v,w)$-CSCK equation on $X$,
where the weights are determined by the fibration data and the extremal affine function
of the compatible K\"ahler class. 
Combining this description with Theorem \ref{main-theorem}, 
we show in Corollary \ref{ext-to-g(v,w)} that, whenever the associated weight $g(v,w)$ is positive, 
the existence of a compatible extremal metric on $Y$ is equivalent to
the existence of a $g(v,w)$-soliton on $X$.
This reduces the existence problem for compatible extremal metrics to
the solvability of a complex Monge–Amp\`ere equation on the Fano fiber.


We use this reduction obtained by Corollary \ref{ext-to-g(v,w)}
for a specific non-toric Fano fiber case. 
We consider
\begin{equation}
X_{3.14}:=\operatorname{Bl}_{p}\operatorname{Bl}_{C}\mathbb P^3,
\end{equation}
where $C$ is a smooth cubic curve contained in a plane $H\subset\mathbb P^3$ and $p\in\mathbb P^3\setminus H$. 
This is the Fano three-fold \textnumero$3.14$
in the article \cite{ACC+23}
which is based on the classifications \cite{Isk77, Isk78, MM81, MM03}. 
It is non-toric and has a maximal algebraic torus isomorphic to $\mathbb C^*$. 
It is also known that the vanishing of the $g$-weighted Futaki invariant is sufficient for
the existence of a $g$-soliton for a smooth positive weight $g$ \cite{LWang25}.
Since the weights associated with the extremal affine function satisfy this vanishing condition, 
this property allows us to obtain explicit existence results for
extremal metrics on the corresponding fibrations.
More precisely, let $B$ be a compact Riemann surface of genus two, 
and consider the simple principal fibration $Y\to B$ with fiber $X_{3.14}$
constructed in Section \ref{semisimple-fibration}.
The total space $Y$ is a non-Fano complex fourfold. 
Its compatible K\"ahler classes $\alpha_c=[\omega_{Y,c}]$ are parametrized by $c>1$,
corresponding to the positive affine functions $x+c$ on the moment polytope $[-1,1]$.
\begin{theorem}[Theorem \ref{2nd-main-thm}]\label{main-theorem2}
    The compatible K\"ahler class $\alpha_c$ on the simple principal fibration $Y$
    with fiber $X_{3.14}$ over $B$
    admits an extremal metric if $c>c_{\mathrm{sol}}\approx1.803$, 
    where $c_{\mathrm{sol}}$ is the unique root greater than $1$ of
    the equation $3200c^3+1033c^2-8466c-6846=0.$
    Moreover, the class $\alpha_{c_*}$ admits a CSCK metric for
    $c_*=2+\frac{\sqrt{482}}{8}\approx4.744.$
\end{theorem}

To the best of the author's knowledge,
this theorem yields new {\it explicit} examples of 
$4$-dimensional extremal K\"ahler manifolds
and a CSCK manifold.

We also examine a ruled surface with fiber $\mathbb P^1$ in Theorem \ref{2nd-main-thm'}
and fibrations with toric Fano fiber in Corollary \ref{2nd-main-thm''}. 
This example illustrates the role of the positivity assumption:
the associated soliton exists in a proper subrange of the K\"ahler classes 
known to admit extremal metrics \cite{ACGT08,Ton98}.
For toric Fano fibers, the weighted barycenter criterion for
the existence of a $g$-soliton \cite{BB13}
gives an alternative proof of the strict-inequality case of the sufficient condition
of Delcroix–Jubert \cite[Theorem 1.1]{DJ23}.
Our approach thereby gives a geometric interpretation of their numerical condition 
in terms of the positivity of the associated soliton weight.

\subsection{Application to Sasaki extremal metrics and Sasaki Mabuchi solitons}
The third result in this article is an application of Theorem \ref{main-theorem}
to generalize Theorem \ref{meta-theorem} in the framework of Sasaki geometry.
Sasaki extremal metrics,
introduced by Boyer-Galicki-Simanca \cite{BGS08},
are the transverse counterparts of Calabi’s extremal  metrics,
and it is natural to ask whether the existence equivalence in Theorem \ref{meta-theorem} extends to this setting. Since the transverse K\"ahler geometry depends on the Reeb vector field, a further question is how the existence of a Sasaki extremal metric and a Sasaki Mabuchi soliton varies with this choice. We establish a Sasakian analogue of Theorem \ref{meta-theorem} on canonical circle bundles over Fano manifolds, including irregular Reeb fields, and obtain an explicit necessary and sufficient condition for the existence of Sasaki Mabuchi solitons on the odd-dimensional sphere $S^{2n+1}$.

To formulate the Sasakian analogue, we first describe the geometric setting in which our results on Fano manifolds can be applied. Let $X$ be an $n$-dimensional Fano manifold, let $T$ be a maximal compact torus in $\operatorname{Aut}(X)$, and fix a $T$-invariant K\"ahler form $\omega_0\in 2\pi c_1(X)$. The unit circle bundle $N_{\omega_0}$ of $K_X$ carries a regular Sasaki structure whose transverse K\"ahler geometry is that of $X$. Its reference Reeb field is the generator $\widehat\chi$ of the fiberwise circle action, normalized to have period $2\pi$. The natural lift of $T$, together with this circle action, generates a torus $\widehat T$, allowing us to vary the Reeb field within the same complex cone.

For the comparison with the Fano K\"ahler setting, we normalize these Reeb fields as
$
\widehat\xi=\xi+\widehat\chi,
$
where $\xi\in\mathfrak t$ satisfies the positivity condition $l_\xi(x):=1+\langle x,\xi\rangle>0$ on the canonically normalized moment polytope $P_X$. 
This condition ensures that $\widehat\xi$ belongs to the Reeb cone. 
By this normalization for $\widehat{\xi}$, the contact form $\eta^{\widehat{\xi}}$ satisfies
$[d\eta^{\widehat{\xi}}]_{B,\widehat\xi}=2\pi c_1^{B,\widehat\xi}(N_{\omega_0})$
and fixes the average of the transverse scalar curvature to be $n$.

We write $\Xi_{\widehat\xi}^{\widehat T}$ for the space of $\widehat T$-invariant transverse K\"ahler potentials for the fixed transverse complex structure and basic K\"ahler class, as defined in Section \ref{regular-quotient}. In this setting, a Sasaki Mabuchi soliton is a transverse K\"ahler metric for which $e^{\rho^T}$ is an affine function of the contact moment map, where $\rho^T$ denotes the normalized transverse Ricci potential. Our first result in Sasaki geometry is the following.

\begin{theorem}[Corollary \ref{Sasaki-ext-sol}]\label{main-theorem3}
    Fix a normalized Reeb field $\widehat\xi=\xi+\widehat\chi$ on $N_{\omega_0}$ as above. Then the following conditions are equivalent.
    \begin{enumerate}
        \item There exists a Sasaki extremal metric in $\Xi_{\widehat\xi}^{\widehat T}$ whose transverse scalar curvature is strictly less than $n+1$.
        \item There exists a Sasaki Mabuchi soliton in $\Xi_{\widehat\xi}^{\widehat T}$.
    \end{enumerate}
\end{theorem}

The curvature bound in this theorem has an intrinsic interpretation 
in terms of the Sasaki Futaki invariant $\mathrm{Fut}_{\widehat{\xi}}$. 
As defined in Section \ref{Sasaki-section},
let $l_{\widehat\xi,\mathrm{ext}}$ be the extremal affine function on the characteristic hyperplane $P_{\widehat\xi}$,
which can be regarded as the dual function of $\mathrm{Fut}_{\widehat{\xi}}$.
This defines the Sasaki Mabuchi constant
\begin{equation}
M_{N_{\omega_0},\widehat\xi}
:=
\sup_{P_{\widehat\xi}}
\bigl(l_{\widehat\xi,\mathrm{ext}}-n\bigr).
\end{equation}
Theorem \ref{main-theorem3} shows that, whenever a Sasaki extremal metric exists, 
the necessary condition $M_{N_{\omega_0},\widehat\xi}<1$
is also sufficient for the existence of a Sasaki Mabuchi soliton. 


The proof of Theorem \ref{main-theorem3} combines the weighted scalar curvature formula of 
Apostolov–Calderbank \cite[Lemma 3]{AC21}, the correspondence between invariant transverse K\"ahler potentials for different Reeb fields established by 
Apostolov–Calderbank–Legendre \cite[Corollary 2.15]{ACL21},
and the transformation formula for transverse Ricci potentials of
Apostolov–Lahdili–Legendre \cite[Lemma 4.3]{ALL25}. These results allow us to express both metric equations as weighted equations on $X$, which is the regular quotient associated with $\widehat\chi$, even when the Reeb field $\widehat\xi$ is irregular. The resulting weights satisfy precisely the relation defining $g(v,w)$, 
so that Theorem \ref{main-theorem} applies.
The same argument establishes an existence equivalence between transverse $g$-extremal metrics and transverse $g$-solitons for smooth positive normalized weights $g$.
See Theorem \ref{Sasaki-g-ext-sol}.

Let us consider the standard odd-dimensional sphere
\begin{equation}
S^{2n+1}=\Set{(w_{0},\dots,w_{n})\in\mathbb C^{n+1} | \sum_{i=0}^{n}|w_{i}|^{2}=1},
\end{equation}
which is the most fundamental Sasaki manifold.
As an application of Theorem \ref{main-theorem3}, 
we determine exactly which normalized Reeb fields in the standard Reeb cone of
$S^{2n+1}$ admit Sasaki Mabuchi solitons. 
In this case, the Reeb fields can be parametrized explicitly 
using the coordinatewise natural action of $\widehat{T}=(S^1)^{n+1}$. 
Let $\xi_a$ be the generator of the $S^1$-action on the $i$-th coordinate.
A normalized Reeb field is written as
\begin{equation}
\xi_a:=\sum_{i=0}^{n}a_i\xi_i,
\end{equation}
where
$a_i>0$ satisifies
$
\sum_{i=0}^{n}a_i=1.
$
Boyer–Galicki–Simanca [17, Theorem 7.8] proved that a Sasaki extremal metric exists for every such $\xi_a$. 
Theorem \ref{main-theorem3} therefore reduces the existence problem for Sasaki Mabuchi solitons to the computation of the Sasaki Mabuchi constant. This yields the following criterion.

\begin{theorem}[Theorem \ref{Mab-sol-sphere}]\label{main-theorem4}
    There exists a $\widehat{T}$-invariant Sasaki Mabuchi soliton 
    with Reeb field $\xi_a$ on $S^{2n+1}$ if and only if
    \begin{equation}
    \min_{0\leq i\leq n}a_i>\frac{1}{n+2}.
    \end{equation}
\end{theorem}

The criterion in Theorem \ref{main-theorem4} is therefore exactly the condition $M_{S^{2n+1},\xi_a}<1$. Thus Sasaki Mabuchi solitons exist on an explicitly determined proper open region of the normalized Reeb cone. This region contains the symmetric point $a_0=\cdots=a_n=1/(n+1)$, corresponding to the unique normalized Sasaki–Einstein Reeb field, and includes both quasi-regular and irregular Reeb fields.

\begin{orga}
This article is organized as follows.
Section~2 reviews weighted CSCK metrics and $g$-solitons
and fixes our conventions. Section~3 introduces the associated
weight $g_{(v,w)}$ through identities relating the corresponding
Futaki-type invariants and energy functionals. Section~4 proves
Theorem~\ref{main-theorem} using a continuity method and estimates
for the weighted Mabuchi functional. Section~5 applies the main
theorem to compatible extremal metrics on semisimple principal
fibrations with Fano fiber and presents explicit examples.
Section~6 establishes the Sasakian existence equivalences through
a regular Fano quotient and determines the normalized Reeb fields
on odd-dimensional spheres that admit Sasaki Mabuchi solitons.
\end{orga}

\begin{ackn}
The author would like to thank Vestislav Apostolov and Yasufumi Nitta for sharing
their research article and for helpful discussions and comments.
He would like to thank Tomoyuki Hisamoto for many stimulating discussions
and valuable comments.
He also would like to thank Eiji Inoue and Abdellah Lahdili for helpful comments.
The application to semisimple principal fibrations presented in Section \ref{Application:fibration} grew out of discussions with Simon Jubert at the K\"ahler Geometry Workshop in Sardinia 2026.
The author is grateful to Simon Jubert for these discussions and to the organizers for making the workshop possible.
He is supported by JSPS Grant-in-Aid for Early-Career Scientists
No.~24K16917.
\end{ackn}

\begin{ai}
ChatGPT-6.1 Sol was used to check computations in Theorem \ref{2nd-main-thm}.
It was also used for editorial assistance in preparing the manuscript.
The author takes full responsibility for 
all mathematical arguments and statements in this paper.
\end{ai}

\section{Preliminaries on $(v,w)$-CSCK metrics and $g$-solitons}
\label{weighted section}
In this section we fix some notation to give precise definitions of
weighted CSCK metrics and soliton-type metrics.
Let $X$ be an $n$-dimensional compact K\"ahler manifold with a
K\"ahler class $\Omega$.
Let $T$ be a maximal compact torus in
the reduced automorphism group $\mathrm{Aut}_{r}(X)$,
and $\frak{t}$ be the Lie algebra of $T$.
We fix a $T$-invariant K\"ahler form $\omega_0\in\Omega$ and a moment
map $m_0$. The moment maps $m_\varphi$ associated with
$\omega_\varphi=\omega_0+\sqrt{-1}\partial\bar\partial\varphi$
are chosen compatibly with this fixed equivariant normalization.
Whenever $X$ is Fano and $\Omega=2\pi c_1(X)$, we use throughout
the canonical normalization determined by the natural
lift of the $T$-action to $K_X^{-1}$. Equivalently, we require
\begin{equation}\label{can-lift}
  \int_X \left(l(m_\varphi)-l(0)\right)
  e^{\rho_\varphi}\omega_\varphi^n=0
\end{equation}
for any affine function $l$ on the dual space $\frak{t}^*$,
where the Ricci potential $\rho_\varphi$ satisfies
\begin{equation}\label{Ricci potential}
  \operatorname{Ric}(\omega_\varphi)-\omega_\varphi
  =\sqrt{-1}\partial\bar\partial\rho_\varphi
  \quad\text{and}\quad
    \int_{X}(e^{\rho_{\phi}}-1)\omega_\phi^n=0.
\end{equation}

It is well known that the image $P_{X}:=m_{\phi}(X)$ is
a convex polytope
independent of the choice of 
$\phi$ in
\begin{equation}
\mathcal{H}^{T}:=\Set{\phi\in C^\infty(X;\mathbb R)^T |
\omega_\phi:=\omega_0+\deldel\phi>0}.
\end{equation}
Fix a strictly positive and smooth function $v$ on $P_{X}$.
One can assume that
the weighted Monge-Amp\`ere measure
\begin{equation}
\MA_{v}(\phi):=\frac{v(m_{\phi})\omega_{\phi}^{n}}{\int_{X}\omega_{0}^{n}}
\end{equation}
defines a probability measure on $X$ without loss of generality,
since the integral $\int_{X}v(m_{\phi})\omega_{\phi}^{n}$
is independent of the choice of $\phi\in\mathcal{H}^{T}$ \cite{Lah19}.
We denote the ordinary Monge-Amp\`ere measure
$\omega_{\phi}^{n}/\int_{X}\omega_{0}^{n}$
by $\MA(\phi)$.

Following \cite{DJL24}, we introduce generalized notions of
the trace and the Laplacian in the weighted setting.

\begin{definition}(\cite[Appendix A]{DJL24})
For any $T$-invariant $(1,1)$-form $\theta$ with a moment map
$m_{\theta}$, the $v$-weighted trace of $\theta$ is defined by
\begin{equation}
\tr_{v,\phi}(\theta):=\tr_{\phi}(\theta)
+\langle (\log v)^{\prime}(m_{\phi}),m_{\theta}\rangle,
\end{equation}
where $\tr_{\phi}(\theta)$ is the ordinary trace defined by the equation
$\tr_{\phi}(\theta)\omega_{\phi}^{n}=n\theta\wedge\omega_{\phi}^{n-1}$.
Equivalently, in a fixed basis $(\xi_{1},\dots,\xi_{r})$ of $\mathfrak{t}$,
\begin{equation}
\tr_{v,\phi}(\theta)=\tr_{\phi}(\theta)+
\frac{1}{v(m_{\phi})}\sum_{j}v_{, j}(m_{\phi})m^{j}_{\theta},
\end{equation}
where $v_{, j}:=\frac{\partial v}{\partial x_{j}}$ and
$m_{\theta}=(m_{\theta}^{1},\dots,m_{\theta}^{r}).$
\end{definition}

\begin{definition}(\cite[Appendix A]{DJL24})
For any smooth function $f$ on $X$, the $v$-weighted Laplacian
for $\phi\in\mathcal{H}^{T}$ is defined by
\begin{equation}
\Delta_{v,\phi}f:=-\frac{1}{v(m_{\phi})}
\overline{\partial}^{*}(v(m_{\phi})\overline{\partial}f)
= \Delta_{\phi} f +\frac{1}{v(m_{\phi})}
\langle\overline{\partial}v(m_{\phi}),\overline{\partial}f\rangle,
\end{equation}
where $\Delta_{\phi}$ is the ordinary (negative) Laplacian for $\phi$.
In particular $\Delta_{v,\phi}f=\tr_{v,\phi}(\deldel f)$.
\end{definition}

It follows from the definition that
the operator $\Delta_{v,\phi}$ is elliptic and 
self-adjoint with respect to
the $v$-twisted Hermitian inner product
$\ip{f_{1}, f_{2}}_{v, \phi}:=\int_{X}f_{1}\overline{f_{2}}\,\MA_{v}(\phi)$
on $C^{\infty}(X;\mathbb{C})$,
and the kernel of $\Delta_{v,\phi}$ consists of the constant functions on $X$.

By considering the weighted Monge-Amp\`ere measure and
the weighted trace, one can define the notions of weighted curvatures.

\begin{definition}
The $v$-weighted Ricci curvature for $\phi\in\mathcal{H}^{T}$ is defined by
the formula
\begin{equation}
\Ric_{v}(\omega_{\phi}):=-\deldel\log\MA_{v}(\phi)
=\Ric(\omega_{\phi})-\deldel\log v(m_{\phi}),
\end{equation}
where $\Ric(\omega_{\phi}):=-\deldel\log\MA(\phi)$
is the ordinary Ricci curvature for $\phi$.
The $v$-weighted scalar curvature for $\phi\in\mathcal{H}^{T}$ is defined by
the formula
\begin{equation}
S_{v}(\phi):=\tr_{v,\phi}\Ric_{v}(\omega_{\phi}).
\end{equation}
\end{definition}

\subsection{Weighted CSCK metrics}
In order to define a weighted CSCK metric,
let $w$ be a smooth function on $P_{X}$ satisfying
\begin{equation}\label{norm-w}
\int_{X}S_{v}(\phi)\MA_{v}(\phi)=\int_{X}w(m_{\phi})\MA_{v}(\phi).
\end{equation}
The integrals on both sides are independent of the choice
of a metric $\phi\in\mathcal{H}^{T}$ \cite{Lah19}.

\begin{definition}(\cite{Lah19})
A metric $\phi\in\mathcal{H}^{T}$ is called a $(v,w)$ constant scalar curvature
K\"ahler metric (a $(v,w)$-CSCK metric for short) if
\begin{equation}
S_{v}(\phi)=w(m_{\phi}).
\end{equation}
\end{definition}

\begin{remark}\label{w-scalar}
Lahdili \cite[Section 2]{Lah19} originally defined the weighted scalar curvature as
$v(m_{\phi})\tr_{v,\phi}\Ric_{v}(\omega_{\phi})$.
We adopted the definition introduced by Boucksom-Jonsson-Trusiani
\cite[Section 3]{BJT24}.
Namely, a $(v,w)$-CSCK metric as we adopted in this article is equivalent to a $(v,vw)$-CSCK
metric in the sense of Lahdili.
\end{remark}

Typical situations in which a weighted CSCK metric appears are as follows.

\begin{itemize}
\item If $v=1$ and $w$ is constant, the weighted CSCK metric
is the ordinary CSCK metric.
\item If $v=1$ and $w$ is an affine function, the weighted CSCK metric
is the extremal K\"ahler metric introduced by Calabi \cite{Ca82}.
\item If $v$ and $w$ are appropriate polynomials, 
the weighted CSCK metric describes Calabi's extremal metric
on the total space of a holomorphic fibration $Y$ with fiber $X$,
called semisimple principal fibration \cite{ACGT11, AJL23, Jub23}.
\item If $v=e^{l}$ for an affine function $l$ and $w=l+a$ for a constant
$a\in\mathbb{R}$ determined by the relation \eqref{w-formula},
the weighted CSCK metric is the $\mu$-CSCK metric,
in the sense of Inoue \cite{Ino22},
extending the notion of K\"ahler-Ricci solitons on Fano manifolds
to general K\"ahler manifolds.
\item If $\Omega=2\pi c_{1}(L)$ for an ample line bundle over $X$, 
$v=l^{-n-1}$ for a positive affine function
and $w=al^{-1}$ for a constant $a\in\mathbb{R}$,
then the weighted CSCK metric describes
a constant scalar curvature Sasaki metric
on the unit circle bundle associated with $(L^{-1})^{\times}$.
In the same setting as above, if one instead takes $w=l^{\prime}l^{-1}$
for an affine function $l^{\prime}$, the weighted CSCK metric
describes a Sasaki extremal metric
on the unit circle bundle associated with $(L^{-1})^{\times}$.
See \cite{AC21, ACL21}.
\item If $v=l^{-2n+1}$ and $w=l^{-2}$ for a positive affine function $l$,
the weighted CSCK metric describes a K\"ahler metric
which is conformal to an Einstein-Maxwell metric \cite{AM19}.
\end{itemize}
\subsection{Soliton-type metrics}
Let us
assume that $X$ is Fano and $\Omega=2\pi c_{1}(X)$.
In order to introduce soliton-type metrics,
fix a strictly positive and smooth function $g$ on $P_{X}$
such that $\MA_{g}(\phi)$ is a probability measure on $X$
for any $\phi\in\mathcal{H}^{T}$.
One can also define the $g$-weighted Ricci potential function
$\rho_{g,\phi}$ by the formulas
\begin{equation}\label{Ricg}
\Ric_{g}(\omega_{\phi})-\omega_{\phi}=\deldel\rho_{g,\phi}
    \quad\text{and}\quad
    \int_{X}(e^{\rho_{g,\phi}}-1)\MA_{g}(\phi)=0.
\end{equation}
It follows from the definitions that
$\rho_{g,\phi}=\rho_{\phi}-\log g(m_{\phi})$.

\begin{definition}(\cite{BW14, HaLi20})
A metric $\phi\in\mathcal{H}^{T}$ is called a $g$-soliton if
\begin{equation}
\Ric_{g}(\omega_{\phi})=\omega_{\phi},
\quad\text{equivalently,}\quad
\rho_{g,\phi}=0.
\end{equation}
\end{definition}

Let $\mu_{\phi}$ be the canonical measure for $\phi\in\mathcal{H}^{T}$
defined by
\begin{equation}\label{can-measure}
\mu_{\phi}=
\frac{e^{-\phi+\rho_{0}}\MA(0)}{\int_{X}e^{-\phi+\rho_{0}}\MA(0)}.
\end{equation}
A metric $\phi\in\mathcal{H}^{T}$ is a $g$-soliton
if and only if it satisfies
the complex Monge-Amp\`ere equation
$\MA_{g}(\phi)=\mu_{\phi}$.
We will use the following formula frequently later in this paper.
\begin{equation}\label{can-meas-formula}
\mu_{\phi}=e^{\rho_{\phi}}\MA(\phi)
=e^{\rho_{g,\phi}}\MA_{g}(\phi).
\end{equation}

A $g$-soliton is a weighted CSCK metric in $2\pi c_{1}(X)$
with respect to appropriate weight functions.

\begin{proposition}{\rm (\cite{AJL23})}\label{g-sol-csck}
A metric $\phi\in\mathcal{H}^{T}$ is a $g$-soliton
if and only if
it is a $(v,w)$-CSCK metric with the weights defined by, for any $x\in P_{X}$,
\begin{equation}\label{w-formula}
v(x)=g(x)\quad\text{and}\quad w(x)=n+\langle(\log g)^{\prime}(x),x\rangle.
\end{equation}
\end{proposition}

\begin{proof}
In general, by taking the $g$-weighted trace of the defining equation
\eqref{Ricg} of the $g$-weighted Ricci potential $\rho_{g,\phi}$,
one has the formula
\begin{equation}\label{LapRicg}
\Delta_{g,\phi}\rho_{g,\phi}
=S_{g}(\phi)-n-\langle(\log g)^{\prime}(m_{\phi}),m_{\phi}\rangle.
\end{equation}
The proposition follows directly from the formula \eqref{LapRicg}.
\end{proof}

Typical situations in which a $g$-soliton appears are as follows.

\begin{itemize}
\item If $g=1$, the $g$-soliton is a K\"ahler-Einstein metric.
\item If $g=e^{l}$ for an affine function $l$,
the $g$-soliton is a K\"ahler-Ricci soliton \cite{TZ00}.
\item If $g$ is a positive affine function, the $g$-soliton
is a Mabuchi soliton \cite{Ma01}.
\item If $X$ is toric Fano and $g$ is an appropriate polynomial, the $g$-soliton
describes a K\"ahler-Einstein metric on a total space of fibration $Y$ with fiber $X$,
called KSM manifolds \cite{Nakg19, NN24}.
\item If $g=l^{-n-2}$ for a positive affine function $l$,
then the $g$-soliton describes a Sasaki-Einstein metric on
the unit circle bundle associated with $(K_{X})^{\times}$ \cite{AJL23}.
\end{itemize}
\section{$g(v,w)$-solitons}
In this section we explain some background for the formulation of $g(v,w)$-solitons
from viewpoints of Futaki-type invariants and energy functionals.

\subsection{Futaki-type invariants}
Let $(X, \Omega)$ be a compact K\"ahler manifold.
For any affine function $l$ on $\mathfrak{t}^{*}$, we define
\begin{equation}
M^{\prime}_{v,w}(l)
:=-\int_{X}l(m_0)\Big(S_{v}(0)-w(m_0)\Big)\MA_{v}(0).
\end{equation}
Lahdili \cite{Lah19} showed that
the quantity $M^{\prime}_{v,w}(l)$ is independent of the choice
of a $T$-invariant metric $\omega_{0}\in\Omega$.
Thus if there exists a $T$-invariant $(v,w)$-CSCK metric in $\Omega$
then $M^{\prime}_{v,w}$ must vanish identically.
The invariant $M^{\prime}_{v,w}$ gives an obstruction to the existence
of a $T$-invariant $(v,w)$-CSCK metric.
We call $M^{\prime}_{v,w}$ the $(v,w)$-Futaki invariant.

Let us assume that $X$ is Fano and $\Omega=2\pi c_{1}(X)$.
Let us consider describing the $(v,w)$-Futaki invariant
using the $v$-weighted Ricci potential $\rho_{v}$.
For this we use the following lemma.
\begin{lemma}\label{Deltal}
Fix a holomorphic vector field $\xi\in\mathfrak{t}$
with the potential function $l(m_{\phi})$ for an affine function $l$.
The following equation holds.
\begin{equation}
\Delta_{v,\phi}l(m_{\phi})+l(m_{\phi})+\xi(\rho_{v,\phi})
=l(0).
\end{equation}
\end{lemma}
\begin{proof}
It suffices to consider the case where $l(m_\phi)$ is not a constant function.
Note that $L_{\xi}\omega_{\phi}=\deldel l(m_{\phi})$
by definition of the moment map.
By the Lie derivative of the defining equation
of the weighted Ricci potential \eqref{Ricg},
\begin{equation}
\Delta_{v,\phi}l(m_{\phi})+l(m_{\phi})+\xi(\rho_{v,\phi})=C
\end{equation}
for a constant $C$ which is determined by the normalization 
\begin{equation}
0=\int_{X}L_{\xi}\Big(e^{\rho_{v,\phi}}\MA_{v}(\phi)\Big)
=\int_{X}\Big(\Delta_{v,\phi}l(m_{\phi})+\xi(\rho_{v,\phi})\Big)
e^{\rho_{v,\phi}}\MA_{v}(\phi).
\end{equation}
In fact,
\begin{equation}
    C=\int_{X}l(m_{\phi})e^{\rho_{v,\phi}}\MA_{v}(\phi)
     =\int_{X}l(m_{\phi})e^{\rho_{\phi}}\MA(\phi).
\end{equation}
By the normalization of the moment map \eqref{can-lift},
$C=l(0)$.
\end{proof}

Then the weight function
\begin{equation}\label{g(v,w)}
g(v,w)(x):=v(x)\left(1+n+\langle(\log v)^{\prime}(x),x\rangle -w(x)\right)
\end{equation}
naturally appears,
as shown in the following proposition.

\begin{proposition}\label{FutMD}
For any affine function $l$ on $\mathfrak{t}^{*}$, we define
\begin{align*}
D^{\prime}_{g(v,w)}(l)&:=-\int_{X}l(m_{0})\Big(\frac{g(v,w)(m_{0})}{v(m_{0})}
-e^{\rho_{v,0}}\Big)\MA_{v}(0) \\
&=-\int_{X}l(m_{0})\Big(g(v,w)(m_{0})-e^{\rho_{0}}\Big)\MA(0).
\end{align*}
Then we have
$M^{\prime}_{v,w}(l)=D^{\prime}_{g(v,w)}(l)$.
\end{proposition}

\begin{proof}
Use the formula \eqref{LapRicg}
and Lemma \ref{Deltal} to show
\begin{align*}
-M^{\prime}_{v,w}(l)&=\int_{X}
-\langle \dbar l(m_{0}),\dbar\rho_{v,0}\rangle\MA_{v}(0)
+\int_{X}l(m_{0})\Big(\frac{g(v,w)(m_{0})}{v(m_{0})}-1\Big)\MA_{v}(0) \\
&= \int_{X}-\xi(\rho_{v,0})\MA_{v}(0)
+\int_{X}l(m_{0})\Big(\frac{g(v,w)(m_{0})}{v(m_{0})}-1\Big)\MA_{v}(0) \\
&=\int_{X}\Big(\Delta_{v,0}l(m_{0})+l(m_{0})
-\int_{X}l(m_{0})e^{\rho_{v,0}}\MA_{v}(0)\Big)\MA_{v}(0) \\
&\quad+\int_{X}l(m_{0})\Big(\frac{g(v,w)(m_{0})}{v(m_{0})}-1\Big)\MA_{v}(0) \\
&=-D^{\prime}_{g(v,w)}(l).
\end{align*}
\end{proof}

In other words, the existence of a $(v,w)$-CSCK metric and that of a $g(v,w)$-soliton
are obstructed by the same Futaki-type invariant for a Fano manifold $X$.
Proposition \ref{FutMD} provides evidence for the validity of the statement
that the existence conditions for
a $(v,w)$-CSCK metric and a $g(v,w)$-soliton are equivalent.

By taking $l=1$ in Proposition \ref{FutMD},
we can see that the normalization condition \eqref{norm-w} of $w$
is equivalent to the equality
\begin{equation}
    \int_Xg(v,w)(m_0)\MA(0)=\int_Xv(m_0)\MA(0).
\end{equation}

\subsection{Energy functionals}
Let $(X,\Omega)$ be a compact K\"ahler manifold.
Analogously to Mabuchi's definition of the Mabuchi functional \cite{Ma86}
by integrating the Futaki invariant \cite{Fu83},
one has a functional on $\mathcal{H}^{T}$
admitting a $(v,w)$-CSCK metric as its critical point
by integrating $M_{v,w}^{\prime}$.
The $(v,w)$-weighted Mabuchi functional
$M_{v,w}:\mathcal{H}^{T}\to\mathbb{R}$ is defined by
\begin{equation}
M_{v,w}(\phi):=-\int_{0}^{1}dt\int_{X}\dot{\phi}_{t}
\Big(S_{v}(\phi_{t})-w(m_{\phi_{t}})\Big)\MA_{v}(\phi_{t}),
\end{equation}
where $\phi_{t}$ is a path in $\mathcal{H}^{T}$ connecting $0$ and $\phi$.
The value $M_{v,w}(\phi)$ is independent of the choice of this path.
Indeed, a so-called Chen-Tian-type formula,
which gives a path-independent expression, is known \cite{BJT24, Lah19}.

Let us assume that $X$ is Fano and $\Omega=2\pi c_{1}(X)$.
Fix a smooth positive function $g$ on $P_{X}$.
Similar to $M_{v,w}$, one has a functional admitting a $g$-soliton
as its critical point.
The $g$-weighted Ding functional $D_{g}:\mathcal{H}^{T}\to\mathbb{R}$
is defined by
\begin{equation}
D_{g}(\phi):=-\int_{0}^{1}dt\int_{X}\dot{\phi}_{t}
\Big(g(m_{\phi_{t}})-e^{\rho_{\phi_{t}}}\Big)\MA(\phi_{t}),
\end{equation}
where $\phi_{t}$ is a path in $\mathcal{H}^{T}$ connecting $0$ and $\phi$.
Put
\begin{equation}\label{E-and-L}
E_{g}(\phi):=\int_{0}^{1}dt\int_{X}\dot{\phi}_{t}\MA_{g}(\phi_{t})
\quad\text{and}\quad
L(\phi):=-\log\int_{X}e^{\rho_{0}-\phi}\MA(0).
\end{equation}
Then $D_{g}=-E_{g}+L$.
These functionals are defined independently of the choice of a path $\phi_{t}$
\cite{BW14, HaLi20}.

The following proposition,
which plays important roles in our proof of Theorem \ref{main-theorem},
may be viewed as an energy functionals version of Proposition \ref{FutMD}.
\begin{proposition}\label{comparison}
For any $\phi\in\mathcal{H}^{T}$, one has
\begin{equation}\label{M=D}
M_{v,w}(\phi)-D_{g(v,w)}(\phi)
=-\int_{X}\rho_{v,\phi}\MA_{v}(\phi)+\int_{X}\rho_{v,0}\MA_{v}(0).
\end{equation}
In particular, 
\begin{equation}\label{M>D}
M_{v,w}(\phi)-D_{g(v,w)}(\phi)\geq \int_{X}\rho_{v,0}\MA_{v}(0).
\end{equation}
\end{proposition}

\begin{proof}
Use the variation formula
$\delta\rho_{v,\phi}=-\Delta_{v,\phi}\delta\phi-\delta\phi+\int_{X}\delta\phi d\mu_{\phi}$,
where $\mu_{\phi}$ is the canonical measure defined as \eqref{can-measure},
to show
\begin{align*}
&\delta\Big(-\int_{X}\rho_{v,\phi}\MA_{v}(\phi)\Big)\\
&=\int_{X}\Big(\Delta_{v,\phi}\delta\phi+\delta\phi-\int_{X}\delta\phi d\mu_{\phi}\Big)
\MA_{v}(\phi)
-\int_{X}\rho_{v,\phi}(\Delta_{v,\phi}\delta\phi)\MA_{v}(\phi) \\
&=\int_{X}\delta\phi\Big[\Big(\frac{g(v,w)(m_{\phi})}{v(m_{\phi})}
+w(m_{\phi})\Big)\MA_{v}(\phi)-d\mu_{\phi}\Big]
-\int_{X}\delta\phi S_{v}(\phi)\MA_{v}(\phi) \\
&= \delta\left( M_{v,w}(\phi)-D_{g(v,w)}(\phi) \right).
\end{align*}
By integration, one obtains the equality \eqref{M=D}.
The inequality \eqref{M>D} follows from
the elementary inequality $e^{x}\geq 1+x$ and the normalization of $\rho_{v,\phi}$,
that is,
\begin{equation}
-\int_{X}\rho_{v,\phi}\MA_{v}(\phi)\geq\int_{X}(1-e^{\rho_{v,\phi}})\MA_{v}(\phi)=0.
\end{equation}
\end{proof}
\section{Proof of Theorem \ref{main-theorem}}
In this section we prove Theorem \ref{main-theorem}.
Our proof uses the coercivity theorems that characterizes
the existence of a canonical metric in terms of a behavior of a corresponding
functional on the space of K\"ahler metrics.
Let $I:\mathcal{H}^{T}\to\mathbb{R}$ be the Aubin's $I$-functional defined by
\begin{equation}
I(\phi):=\int_{X}\phi\left(\MA(0)-\MA(\phi)\right),
\end{equation}
which plays a role of a distance function in $\mathcal{H}^{T}$.
In the following theorem, we say that a functional $F:\mathcal{H}^{T}\to\mathbb{R}$ is $T^{\mathbb{C}}$-coercive if
there exists positive constants $\e, C >0$ such that the inequality
\begin{equation}
F(\phi)\geq\e\inf_{\sigma\in T^{\mathbb{C}}}I(\sigma[\phi])-C
\end{equation}
holds for any $\phi\in\mathcal{H}^{T}$,
where $\sigma[\phi]\in\mathcal{H}^{T}$ is defined by 
$\sigma^{*}\omega_{\phi}=\omega_{0}+\deldel\sigma[\phi]$.
\begin{theorem}\label{coercivity}
\begin{itemize}
We have the following equivalences.
\item {\rm (\cite{DJL24, DJL25, HanLiu25, JP26})}
There exists a $(v,w)$-CSCK metric in $\mathcal{H}^{T}$ if and only if
the weighted Mabuchi functional $M_{v,w}$ is $T^{\mathbb{C}}$-coercive.
\item {\rm (\cite{HaLi20, JP26})} Let $X$ be a Fano manifold and $\Omega=2\pi c_{1}(X)$.
There exists a $v$-soliton in $\mathcal{H}^{T}$ if and only if
the weighted Ding functional $D_{v}$ is $T^{\mathbb{C}}$-coercive.
\end{itemize}
\end{theorem}

Let us focus on the case where $X$ is Fano and $\Omega=2\pi c_{1}(X)$
to prove Theorem \ref{main-theorem}.
The proof of one direction is based on the comparison of
the energy functionals $M_{v,w}$ and $D_{g(v,w)}$.
In fact it follows from Theorem \ref{coercivity} and
the inequality \eqref{M>D} in Proposition \ref{comparison}
that the existence of a $g(v,w)$-soliton implies that of a $(v,w)$-CSCK metric.
As for the converse direction, there is currently no known method 
to directly deduce the coercivity of $D_{g(v,w)}$ from that of $M_{v,w}$.

Our approach is to construct a $g(v,w)$-soliton 
by a continuity method under the assumption on
the $T^{\mathbb{C}}$-coercivity of $M_{v,w}$.
Let us consider the following PDE for $\phi\in\mathcal{H}^{T}$:
\begin{equation}\label{conti-eq}
\MA_{g(v,w)}(\phi)=e^{-t\phi+\rho_{0}}\MA(0),
\end{equation}
where $t\in [0,1]$.
A solution of the equation \eqref{conti-eq} at $t=1$ is nothing but a $g(v,w)$-soliton.
Let $\mathcal{T}$ be the set of $t\in[0,1]$ such that the equation \eqref{conti-eq}
has a solution in $\mathcal{H}^T$.
According to \cite[Section 3]{HN24},
the set $\mathcal{T}$ is non-empty and open in general.
Moreover, according to \cite[Section 5]{HN24}
together with the regularity argument in \cite[Proposition 3.2]{JP26},
if for any $\e\in(0,1)$ there exists a constant $C>0$ independent of $t$ such that
\begin{equation}
\inf_{\sigma\in T^{\mathbb{C}}}I(\sigma[\phi_{t}])\leq C
\end{equation}
holds for any solution $\phi_{t}\in\mathcal{H}^{T}$ of \eqref{conti-eq}
at any $t\in\mathcal{T}\cap(\e,1)$,
then the set $\mathcal{T}$ is closed,
which implies $1\in\mathcal{T}$.
In view of the assumption on the $T^{\mathbb{C}}$-coercivity of $M_{v,w}$,
the rest of our proof of Theorem \ref{main-theorem} reduces to
the following estimate.

\begin{theorem}\label{bound-of-M}
Fix any $\e\in(0,1)$.
Let $\phi_t$ be a $T$-invariant solution of \eqref{conti-eq} at $t\in(\e,1)$.
There exists a constant $C>0$ independent of $\e$ and $t$ satisfying
\begin{equation}\label{estimate-of-M}
M_{v,w}(\phi_t)\leq C\e^{-1}.
\end{equation}
\end{theorem}

In the rest of this section we prove Theorem \ref{bound-of-M}.
Let $\phi_{t}\in\mathcal{H}^T$ be a solution of the equation \eqref{conti-eq} at $t$.
In order to obtain the upper bound of $M_{v,w}(\phi_{t})$, 
it suffices to control $D_{g(v,w)}(\phi_{t}) -\int_{X}\rho_{v,\phi_{t}}\MA_{v}(\phi_{t})$
from above by the formula \eqref{M=D} in Proposition \ref{comparison}.
Put $\omega_{\phi_{t}}:=\omega_{t}$,
$\rho_{t}:=\rho_{\phi_{t}}$, $\rho_{v,t}:=\rho_{v,\phi_{t}}$,
$v_{t}:=v(m_{\phi_{t}})$ and $g_{t}:=g(v,w)(m_{\phi_{t}})$ for simplicity.

By definition of the weighted Ding functional $D_{g(v,w)}$ and \cite[Proposition 4.4]{HN24},
\begin{equation}\label{Dg-estimate}
D_{g(v,w)}(\phi_{t})=-E_{g(v,w)}(\phi_{t})+L(\phi_{t})\leq L(\phi_{t}).
\end{equation}
On the other hand, we observe the following to control 
$-\int_{X}\rho_{v,\phi_{t}}\MA_{v}(\phi_{t})$.

\begin{lemma}\label{rho-eq}
Let $\phi_{t}\in \mathcal{H}^T$ be a solution of \eqref{conti-eq} at $t$. 
One has the formula
\begin{equation}\label{rho-formula}
\rho_{v,t}=-(1-t)\phi_{t}+\log\frac{g_{t}}{v_{t}}+L(\phi_{t}),
\end{equation}
where $L$ is the functional defined in \eqref{E-and-L}.
\end{lemma}
Note that by the compactness of $X$,
there exists a uniform constant $C>0$ independent of
the choice of $\phi\in\mathcal{H}^{T}$ such that
\begin{equation}
C^{-1}\leq \frac{g(v,w)(m_{\phi})}{v(m_{\phi})} \leq C.
\end{equation}

\begin{proof}
By definition of the weighted Ricci curvature,
\begin{equation}\label{wRic-formula}
\Ric_{v}\omega_{t}=\deldel\log\frac{g_{t}}{v_{t}}+\Ric_{g(v,w)}(\omega_{t}).
\end{equation}
By the formula \eqref{wRic-formula}, definition of the $v$-weighted Ricci potential 
and the equation \eqref{conti-eq},
\begin{equation}
\deldel\rho_{v,t}=\deldel\left(\log\frac{g_{t}}{v_{t}}-(1-t)\phi_{t}\right).
\end{equation}
Thus $\rho_{v,t}=-(1-t)\phi_{t}+\log g_{t}/v_{t}+c_{t}$
for some constant $c_{t}$ depending on $t$.
The constant $c_{t}$ is determined by the normalization of $\rho_{v,t}$
and the equation \eqref{conti-eq}.
In fact $c_{t}=L(\phi_{t})$.
This completes the proof.
\end{proof}

Integrating the both side of the relation in Lemma \ref{rho-eq},
one thus obtains
\begin{equation}
-\int_{X}\rho_{v,\phi_{t}}\MA_{v}(\phi_{t})
\leq
(1-t)\int_{X}\phi_{t}\MA_{v}(\phi_{t}) -\log\left(\inf_{P_{X}} \frac{g(v,w)}{v}\right)-L(\phi_t).
\end{equation}

\begin{lemma}\label{int-phi}
For a fixed $\e\in (0,1)$,
a solution $\phi_{t}$ at $t>\e$ satisfies
\begin{equation}
\int_{X}\phi_{t}\MA_{v}(\phi_{t})
\leq e^{\sup\rho_{0}-1}\left(\inf_{P_{X}} \frac{g(v,w)}{v}\right)^{-1}\e^{-1}.
\end{equation}
\end{lemma}

\begin{proof}
Using the equation \eqref{conti-eq} of the continuity method,
one obtains
\begin{align*}
\int_{X}\phi_{t}\,\MA_{v}(\phi_{t})
&=
\int_{X}\phi_{t}\frac{v_{t}}{g_{t}}\,\MA_{g}(\phi_{t})\\
&\leq
\int_{\{\phi_{t}>0\}}\phi_{t}\frac{v_{t}}{g_{t}}e^{-t\phi_{t}+\rho_{0}}
\,\MA(0)\\
&\leq
e^{\sup\rho_{0}}\left(\inf_{P_{X}} \frac{g(v,w)}{v}\right)^{-1}
\int_{\{\phi_{t}>0\}}\phi_{t}e^{-t\phi_{t}}\MA(0).
\end{align*}
Note that for any $t>\e$, the function $\mathbb{R}_{>0}\ni x\mapsto xe^{-tx}$
is bounded from above by the constant $(e\e)^{-1}$.
This completes the proof.
\end{proof}

Putting together the above arguments,
Theorem \ref{bound-of-M} follows
from the relation \eqref{M=D}, the estimate \eqref{Dg-estimate} and Lemma \ref{int-phi}.
This completes the proof of Theorem \ref{main-theorem}.
\section{$(v,w)$-CSCK metrics and $g(v,w)$-solitons in geometry of semisimple principal fibrations
with Fano fiber}\label{Application:fibration}
In this section, we apply Theorem \ref{main-theorem} to the construction of
CSCK metrics and extremal metrics on semisimple principal fibrations with Fano fiber. 
We consider compatible K\"ahler classes whose restriction to the fiber is anticanonical. 
The equation defining an extremal metric for compatible classes is then a weighted CSCK equation on the fiber \cite{AJL23}. 
Under an explicit positivity condition, its solvability is equivalent to the existence of a $g(v,w)$-soliton on the fiber. This approach applies even when the total space is not Fano. We illustrate it with ruled surfaces and fibrations with toric and non-toric Fano fibers. The non-toric case yields extremal metrics on a non-Fano complex fourfold. We also identify a K\"ahler class on this fourfold admitting a CSCK metric.

We begin by recalling the definition of an extremal metric and the associated extremal affine function.
Let $Y$ be a compact K\"ahler manifold
with a K\"ahler class $\alpha_Y$.
Let $K$ be a maximal torus in the reduced automorphism group
$\mathrm{Aut}_{\mathrm{red}}(Y)$,
and for a $K$-invariant K\"ahler metric $\omega\in\a_Y$,
let $m_\omega:Y\to\frak{k}^*$ be the corresponding moment map,
where $\frak{k}$ is the Lie algebra of $K$.
We say that a $K$-invariant metric $\omega\in\a_Y$ is extremal
if
\begin{equation}
S(\omega)=l_{\rm ext}(m_{\omega}),
\end{equation}
where $l_{\rm ext}$ is the unique affine function on $\frak{k}^*$,
called {\it the extremal affine function},
such that
\begin{equation}\label{def-l_ext}
    \int_Y l(m_\omega)\left( S(\omega)-l_{\rm ext}(m_\omega)\right)
    \MA(\omega)=0
\end{equation}
for any affine function $l$ on $\frak{k}^*$.

\subsection{Semisimple principal fibrations}
\label{semisimple-fibration}

In this section, we explain an application of
Theorem \ref{main-theorem} to semisimple principal fibrations.
We follow the construction of \cite{AJL23,DJ23}, but use throughout
the convention for the scalar curvature and weighted scalar
curvature adopted in the present paper.

Let $T$ be a compact torus with Lie algebra $\mathfrak t$.
Let $(X,\omega_X)$ be an $n_X$-dimensional compact K\"ahler manifold equipped
with a Hamiltonian holomorphic $T$-action.
We denote by
$
    m:= m_{\omega_X}: X\to\mathfrak t^*
$
the moment map,
and put
$
    P_X:=m(X).
$

Let
$
    B=\prod_{a=1}^k(B_a,\omega_a)
$
be a product of CSCK Hodge manifolds. 
We put\footnote{Our 
scalar curvature convention
differs by a factor of two from that used in \cite{AJL23, DJ23}}
$n_a:=\dim_{\mathbb C}B_a$ and
    $s_a:=S(\omega_a)
    =\tr_{\omega_a}\Ric(\omega_a)$.
Let
$
    \pi_B: Q\to B
$
be a principal $T$-bundle endowed with a connection
$\theta\in\Omega^1(Q)\otimes\mathfrak t$ whose curvature is given by
\begin{equation}
    d\theta
    =
    \sum_{a=1}^k
    \pi_B^*\omega_a\otimes p_a,
\end{equation}
where $p_a\in\mathfrak t$.
We identify $p_a\in\mathfrak t$ with the linear function
$
    p_a(x):=\langle x,p_a\rangle,
$
where $x\in\mathfrak t^*$.
Let us consider
\begin{equation}
    Y:=(X\times Q)/T,
\end{equation}
where the $T$-action on $(X\times Q)$ is given by
$\gamma\cdot(x,q):=(\gamma\cdot x,\gamma^{-1}\cdot q)$
for any $\gamma\in T$ and $(x,q)\in X\times Q$.
In fact, $Y$ is a fiber bundle over $B$ with fiber $X$.
We call $Y$ {\it a semisimple principal fibration with fiber $X$}.
When $k=1$, we call $Y\to B$ a {\it simple principal fibration},
following the terminology of \cite[Remark~4.4]{DJ23}.
The $T$-action on $X$ defines an action on $X\times Q$
by the natural action on the first factor.
This action commutes with the above $T$-action on $X\times Q$
and thus descends to a $T$-action on $Y$.

Choose constants $c_a\in\mathbb R$ such that
$
    p_a(m)+c_a>0
$
on $X$ for every $a\in\{1,\dots,k\}$. 
The $2$-form on $X\times Q$ given by
\begin{equation}\label{omega_Y}
    \widetilde\omega_Y
    :=
    \omega_X
    +
    \sum_{a=1}^k
    \bigl(p_a\circ m+c_a\bigr)\pi_B^*\omega_a
    +
    \langle dm\wedge\theta\rangle
\end{equation}
is basic with respect to the $T$-action on $X\times Q$ and descends to
a K\"ahler form on $Y$ \cite{AJL23}, which we denote by $\omega_Y$.
Here $\langle dm\wedge\theta\rangle$ is a $2$-form understood as
the contraction of the $\frak t$-valued $1$-form $\theta$ and
the $\frak t^*$-valued $1$-form $dm$.
Moreover, $\omega_Y$ is invariant under the $T$-action on $Y$ \cite{AJL23}.
The K\"ahler metric $\omega_Y$ is called {\it a compatible K\"ahler metric}.
A K\"ahler class in $H^2(Y;\mathbb R)$ containing a compatible K\"ahler metric
is called {\it a compatible K\"ahler class}.
The constants $c_a$ parametrize the compatible
K\"ahler classes $[\omega_Y]$.

By construction of a semisimple principal fibration $Y$,
there is an embedding of the space of $T$-invariant smooth function
$C^\infty(X)^T$ on $X$ in the space of $T$-invariant smooth function
$C^\infty(Y)^T$ on Y.
The scalar curvature of a compatible K\"ahler metric can be
expressed in terms of the weighted scalar curvature of the fiber.
For any $x\in P_X$, put
\begin{equation}\label{formula-v}
    v(x)
    :=
    \prod_{a=1}^k
    \bigl(p_a(x)+c_a\bigr)^{n_a}.
\end{equation}
According to \cite{ACG06},
the scalar curvature $S(\omega_Y)$ of $\omega_Y$ is in
$C^\infty(X)^T\subset C^\infty(Y)^T$ and
has the following formula
(Recall Remark \ref{w-scalar} for the convention of the scalar curvature);
\begin{equation}\label{compati-S}
    S(\omega_Y)
    =
    \sum_{a=1}^k
    \frac{s_a}{p_a(m)+c_a}
    +
    S_v(\omega_X).
\end{equation}


\subsection{Extremal metrics on semisimple principal fibrations with Fano fiber}
Let $(X,\omega_X)$ be a Fano manifold as above, and
$T\subset\Aut(X)$ be a maximal compact torus.
Let $(Y, \omega_Y, T)$ be a semisimple principal fibration with fiber
$(X,\omega_X,T)$,
and $\a_Y=[\omega_Y]$ is a compatible K\"ahler class.
In general, the torus $T$ acting on $Y$ is not maximal in $\mathrm{Aut}_{\rm red}(Y)$.
By \cite[Lemma 5.14]{AJL23}, any compatible K\"ahler metric in $\a_Y$
is invariant under a maximal torus $K_Y\subset\mathrm{Aut}_{\rm red}(Y)$
containing $T$ satisfying the following property;
there exists a maximal torus $K_B\subset\mathrm{Aut}_{\rm red}(B)$
such that $\omega_B:=\sum_{a=1}^k\omega_a$ is $K_B$-invariant
and the exact sequence
\begin{equation}\label{exact seq}
    0\to\frak{t}\to\frak{k}\to\frak{k}_B\to 0
\end{equation}
holds,
where $\frak{t}:=\mathrm{Lie}(T)$, $\frak{k}:=\mathrm{Lie}(K_Y)$,
and $\frak{t}_B:=\mathrm{Lie}(K_B)$.
Let $\Aff(\frak{t^*})$ (resp. $\Aff(\frak{k}^*)$) be
the space of affine functions on $\frak{t}^*$ (resp. on $\frak{k}^*$).
Following \cite[Section 3.5]{ACGT11} or \cite[Lemma 5.14]{AJL23}, we have

\begin{proposition}
    The extremal affine function $l_{\rm ext}$ for
     a compatible K\"ahler class is in the image of the inclusion
     $\Aff(\frak{t}^*)\to\Aff(\frak{k}^*)$ induced by \eqref{exact seq}.
\end{proposition}

Indeed, by the formula \eqref{compati-S} and the exact sequence
\eqref{exact seq},
the condition \eqref{def-l_ext} is equivalent to
\begin{equation}\label{(v,w)_ext}
    \int_Xl(m_{\omega_X})\left( S_v(\omega_X)-w(m_{\omega_X}) \right)
    \MA_v(\omega_X)=0
\end{equation}
for any $l\in\Aff(\frak{t}^*)$,
where
\begin{equation}\label{formula-w}
    w(x):=l_{\mathrm{ext}}(x)-\sum_{a=1}^k
    \frac{s_a}{p_a(x)+c_a}.
\end{equation}
Thus the equation for an extremal metric on $Y$
is reduced to that for a $(v,w)$-CSCK metric on $X$.
\begin{theorem}[\cite{AJL23}]\label{ext-(v,w)}
    The $K_Y$-invariant compatible K\"ahler metric $\omega_Y$ on $Y$ is extremal
    if and only if the $T$-invariant K\"ahler metric $\omega_X$ on $X$
    is a $(v,w)$-CSCK metric,
    where the weights $v$ and $w$ are given by \eqref{formula-v} and \eqref{formula-w}
    respectively.
\end{theorem}

Let us consider the case where the fiber $X$ is Fano.
By Proposition \ref{FutMD} and the normalization \eqref{can-lift} of the moment map, 
the condition \eqref{(v,w)_ext} is equivalent to
\begin{equation}\label{(v,w)_ext'}
    -l(0)+\frac{1}{\int_X\MA_v(\omega_X)}\int_{X}l(m_{\omega_X})g(v,w)(m_{\omega_X})
    \MA(\omega_X)=0
\end{equation}
for any $l\in\Aff(\frak{t}^*)$.
Moreover, by Theorem \ref{main-theorem} and \ref{ext-(v,w)},
the existence of an extremal metric on $Y$
is reduced to that of a $g(v,w)$-soliton on $X$.
\begin{corollary}\label{ext-to-g(v,w)}
    Let $X$ be a Fano manifold with a maximal torus $T\subset\Aut(X)$
    and $\omega_X\in2\pi c_1(X)$ a $T$-invariant K\"ahler metric.
    Let $(Y, \omega_Y, T)$ be a semisimple principal fibration with fiber
    $(X,\omega_X,T)$ and $\a_Y=[\omega_Y]$ a compatible K\"ahler class.
    Assume that the associated weight $g(v,w)$
    is positive 
    on the moment polytope $P_X$, 
    where the weights $v$ and $w$ are given by \eqref{formula-v} and \eqref{formula-w}
    respectively.
    Then the following are equivalent.
    \begin{enumerate}
        \item There exists an extremal compatible K\"ahler metric in $\a_Y$.
        \item There exists a $g(v,w)$-soliton in $2\pi c_1(X)$.
    \end{enumerate}
\end{corollary}
We emphasize again that there is no formula connecting
an extremal metric on $Y$ and a $g(v,w)$-soliton on $X$
such as the relation \eqref{compati-S}.

We compute the associated soliton weight $g(v,w)$.
Since
\begin{equation}\label{log-v}
    \log v(x)
    =
    \sum_{a=1}^k
    n_a\log\bigl(p_a(x)+c_a\bigr),
\end{equation}
we have
\begin{equation}
    \langle d\log v(x),x\rangle
    =
    \sum_{a=1}^k
    n_a\frac{p_a(x)}{p_a(x)+c_a}.
\end{equation}
Put $n_X:=\dim_{\mathbb C}X$ and $n_Y:=\dim_{\mathbb C}Y$.
Consequently,
\begin{equation}\label{fiber-g(v,w)}
\begin{aligned}
    g(v,w)(x)
    &=
    v(x)\left(
    n_X+1
    +
    \sum_{a=1}^k
    n_a\frac{p_a(x)}{p_a(x)+c_a}
    -l_{\mathrm{ext}}(x)
    +
    \sum_{a=1}^k
    \frac{s_a}{p_a(x)+c_a}
    \right)\\
    &=
    v(x)\left(
    n_Y+1-l_{\mathrm{ext}}(x)
    +
    \sum_{a=1}^k
    \frac{s_a-n_ac_a}{p_a(x)+c_a}
    \right).
\end{aligned}
\end{equation}
In order to apply Corollary \ref{ext-to-g(v,w)},
we need to verify the positivity
of
the function \eqref{fiber-g(v,w)} on $P_X$.
Note that $v$ in \eqref{formula-v} is automatically positive,
since each $p_a+c_a$ is a positive affine function on $P_X$.

\subsubsection{The Fano total space case}
Furthermore, assume that $Y$ is Fano and that the compatible
K\"ahler class is anticanonical, i.e.
$[\omega_Y]=2\pi c_1(Y).$
By \cite[Proposition 6.8]{DJ25},
$c_a[\omega_a]=2\pi c_1(B_a)$ with $c_a>0$ for each $a$.
Since the base metric $\omega_a$ is K\"ahler-Einstein,
its scalar curvature satisfies $s_a=n_ac_a$.
Thus the associated weight $g(v,w)$ is reduced to
\begin{equation}
g(v,w)(x)=v(x)\left(n_Y+1-l_{\rm ext}(x)\right).
\end{equation}
Similarly to Theorem \ref{ext-(v,w)},
the equation of a $g$-soliton on $Y$ is reduced to
that of a $(vg)$-soliton on $X$ \cite{AJL23}.
In particular, we have

\begin{theorem}[\cite{AJL23}]\label{sol-to-sol}
    Assume the affine function $n_Y+1-l_{\rm ext}$ is positive
    on the moment polytope $P_Y$ for the $K_Y$-action on $Y$.
    Then the $K_Y$-invariant compatible K\"ahler metric
    $\omega_Y\in 2\pi c_1(Y)$ is a $(n_Y+1-l_{\rm ext})$-soliton
    if and only if
    the $T$-invariant K\"ahler metric $\omega_X\in 2\pi c_1(X)$
    is a $v(n_Y+1-l_{\rm ext})$-soliton.
\end{theorem}

Put $M_Y:=\sup_{P_Y}(l_{\rm ext}-n_Y)$.
The positivity of $n_Y+1-l_{\rm ext}$ on $P_Y$ is equivalent to
the condition $M_Y<1$.
Following \cite{HN24, Ma01}, we call a $(n_Y+1-l_{\rm ext})$-soliton
{\it a Mabuchi soliton}, and call $M_Y$ {\it the Mabuchi constant} of $Y$.
The condition $M_Y<1$ is necessary for the existence of a Mabuchi soliton
as originally observed by Mabuchi \cite{Ma01}.
When $M_Y<1$, Corollary \ref{ext-to-g(v,w)} and Theorem \ref{sol-to-sol}
prove that the existence of an extremal metric in $2\pi c_1(Y)$
is equivalent to that of a Mabuchi soliton in $2\pi c_1(Y)$.
This is a special case of the general equivalence for Fano manifolds
stated in Theorem \ref{meta-theorem}.

\subsection{Examples of semisimple principal fibrations with Fano fiber
admitting extremal compatible K\"ahler metric}

We construct explicit examples of semisimple principal fibrations with Fano fiber
whose total space is not necessarily Fano and
admits an extremal compatible K\"ahler metric. 
Our approach is based on Corollary \ref{ext-to-g(v,w)},
which reduces the problem to verifying 
the positivity 
of the weight $g(v,w)$ given by \eqref{fiber-g(v,w)}.
We discuss the following three settings separately: 
(i) the fiber is $\mathbb{P}^1$ and the base is a compact Riemann surface of genus two; 
(ii) the fiber is a toric Fano manifold;
and (iii) the fiber is a weight-insensitive K-polystable Fano three-fold
in the sense of \cite{De24},
and the base is a compact Riemann surface of genus two. 
In case (ii), we provide a geometric interpretation of the results of \cite{DJ23} from the viewpoint of $g(v,w)$-solitons, together with an alternative proof of part of their results. 
The four-dimensional extremal K\"ahler manifolds,
including a CSCK manifold,
constructed in case (iii) appear to be new, to the best of the author's knowledge.

\subsubsection{$\mathbb{P}^1$ fiber case}

Let $B$ be a compact Riemann surface of genus $2$, equipped with
a K\"ahler-Einstein metric $\omega_B$ normalized by
$S(\omega_B)=-2.$
Let $L\to B$ be a holomorphic line bundle of degree $-1$ and
consider the ruled surface
\begin{equation}
    Y=\mathbb P_B(\mathcal O_B\oplus L)
\end{equation}
as a simple principal fibration with fiber
$X=\mathbb P^1.$
Notice that the total space $Y$ is not Fano. 
The K\"ahler classes on $Y$ that admit extremal metrics have been completely classified \cite{Sze09, Ton98}.
In what follows, we compare this classification with the existence of extremal metrics
obtained by Corollary \ref{ext-to-g(v,w)}.

We adopt the normalization $[\omega_X]=2\pi c_1(\mathbb P^1)$
and use the moment polytope $P=[-1,1]$ with respect to
the $S^1$-action $e^{\sqrt{-1}t}\cdot[z_0,z_1]=[z_0,e^{\sqrt{-1}t}z_1]$.
The fibration data are given by the affine function $p(x)=x$.
For each $c>1$, the condition
$p(x)+c>0$ for any $x\in P$
defines a compatible K\"ahler class $[\omega_{Y,c}]$ on $Y$
by the formula \eqref{omega_Y}. 

\begin{lemma}
    Let $e_0$ and $e_\infty$ be the Poincar\'e duals of the $0$-section
    $S_0$ and the $\infty$-section $S_\infty$ of $Y$ respectively.
    Then we have
    \begin{equation}\label{[omega_Y]}
        [\omega_{Y,c}]=2\pi\left( -(c-1)e_0+(c+1)e_\infty \right).
    \end{equation}
    Moreover, up to scaling, every K\"ahler class on $Y$ is
    of the form $[\omega_{Y,c}]$ for some $c>1$.
\end{lemma}
\begin{proof}
    By the normalizations of 
    $\omega_B$,
    we have
        $\int_B\omega_B=2\pi$.
    Since, by definition of $\omega_{Y,c}$ in \eqref{omega_Y},
    \begin{equation}
    \omega_{Y,c}|_{S_0}=(c-1)\omega_B
    \quad\text{and}\quad
    \omega_{Y,c}|_{S_\infty}=(c+1)\omega_B,
    \end{equation}
    we also have
    \begin{equation}
        \int_{S_0}\omega_{Y,c}|_{S_0}=2\pi(c-1),
        \quad\text{and}\quad
        \int_{S_\infty}\omega_{Y,c}|_{S_\infty}=2\pi(c+1).
    \end{equation}
    By the fact that $\dim H^{1,1}(Y;\mathbb R)=2$ and
    the intersection formulas
    $S_0^2=-1$, $S_\infty^2=1$ and $S_0\cdot S_\infty=0$,
    we thus obtain \eqref{[omega_Y]},
    and we see that every K\"ahler class is
    of the form $[\omega_{Y,c}]$ for some $c>1$
    up to scaling.
\end{proof}

\begin{remark}
    Let $F$ be a fiber in the bundle $Y$.
    Then $S_\infty=S_0+F$.
    The K\"ahler class $[\omega_{Y,c}]$ corresponds to
    the $\mathbb R$-divisor
    \begin{equation}
        \frac{2\pi}{c-1}\left(F+\frac{2}{c-1}S_\infty\right).
    \end{equation}
    By setting $m=2/(c-1)$, this convention agrees,
up to scaling, with that used in \cite{Sze09}.
\end{remark}

Let
\begin{equation}
    l_c(x):=a_c+b_cx
\end{equation}
denote the extremal affine function associated
with the compatible K\"ahler class $[\omega_{Y,c}]$.
In view of \eqref{formula-v} and \eqref{formula-w},
define the weight functions on $P=[-1,1]$ as
\begin{equation}
v_c(x):=x+c
\quad\text{and}\quad
w_c(x):=l_c(x)+\frac{2}{x+c}.
\end{equation}
Then by the formula \eqref{fiber-g(v,w)},
\begin{equation}
    g(v_c,w_c)(x)=(x+c)\left(3-a_c-b_cx-\frac{2+c}{x+c} \right).
\end{equation}
In order to observe the existence of an extremal metric in $[\omega_{Y,c}]$
by applying Corollary \ref{ext-to-g(v,w)}, 
we determine below the range of $c>1$
for which a $g(v_c,w_c)$-soliton exists on $X=\mathbb P^1$.

\begin{theorem}\label{2nd-main-thm'}
    There exists a $g(v_c,w_c)$-soliton on $X=\mathbb P^1$
    if and only if $c>c_{\rm sol}\approx 1.584$,
    where $c_{\rm sol}$ is the unique root greater than $1$
    of the equation $3c^3-5c-4=0$.
    Consequently, there exists an extremal metric in $[\omega_{Y,c}]$
    if $c>c_{\rm sol}$.
\end{theorem}
\begin{proof}
The extremal affine function is characterized by
the condition \eqref{(v,w)_ext'}.
In the polytope description this amounts to
that the barycenter of $g(v_c,w_c)(x)dx/|P|$ is $0$,
that is,
\begin{equation}
   -l(0)+\frac{1}{|P|}\int_Pl(x)g(v_c,w_c)(x)dx=0
\end{equation}
for any affine function $l$ on $P$,
where 
\begin{equation}
    |P|=\int_Pg(v_c,w_c)(x)\,dx=\int_Pv_c(x)\,dx=2c.
\end{equation}
Taking $l=1$ and $l=x$, respectively, gives
\begin{equation}
    ca_c+\frac13b_c=c-2
    \quad\text{and}\quad
    a_c+cb_c=3.
\end{equation}
Solving these equations, we obtain
\begin{equation}
    a_c
    =
    \frac{3(c^2-2c-1)}{3c^2-1}
    \quad\text{and}\quad
    b_c
    =
    \frac{6(c+1)}{3c^2-1}.
\end{equation}
Thus
\begin{equation}
    g(v_c,w_c)(x)
    =
    \frac{
        3c^3+c+2-6(c+1)x^2
    }{
        3c^2-1
    }.
\end{equation}

Since $c>1$, the minimum of $g(v_c,w_c)$ on $P$
is attained at $x=\pm 1$, that is,
\begin{equation}
    \min_Pg(v_c,w_c)=g(v_c,w_c)(\pm 1)
    =\frac{3c^3-5c-4}{3c^2-1}.
\end{equation}
It follows that $g(v_c,w_c)$ is strictly positive on $P$
if and only if $c>c_{\rm sol}$.
Therefore, if $c>c_{\rm sol}$,
then there exists a $g(v_c,w_c)$-soliton on $X$
by analysis of toric complex Monge-Amp\`ere equation \cite{BB13}.
\end{proof}

In the above proof, we can see that the extremal affine function $l_c$
is not a constant function on $P$ for any $c>1$.
Thus, $[\omega_{Y,c}]$
does not admit a CSCK metric for any $c>1$.

By \cite{ACGT08, Ton98} (see also \cite{Sze09}),
$[\omega_{Y,c}]$ admits an extremal metric if and only if
$c>c_{\rm ext}\approx 1.106$, where $m=2/(c_{\rm ext}-1)$
is the unique root greater than $1$ of the equation
$m^4-16m^3-52m^2-48m-12=0$.
The above observation shows Corollary \ref{ext-to-g(v,w)}, 
applied via $g(v_c,w_c)$-solitons, 
does not allow us to characterize all K\"ahler classes admitting an extremal metric. 
Despite this limitation, we emphasize that
this approach has the advantage of reducing the existence problem for extremal metrics to the solvability of a certain complex Monge-Amp\`ere equation.

\subsubsection{Toric fiber case}
Let $(X, \omega_X)$ be a toric Fano manifold with $[\omega_X]=2\pi c_1(X)$,
and $P$ be the moment polytope, 
normalized so that $0\in P$ is its unique interior lattice point.
As in Section \ref{semisimple-fibration},
let us consider the semisimple principal fibration $(Y,\omega_Y)$ with fiber $X$
over a product of CSCK Hodge manifolds $B$.
Let $v$ and $w$ be weight functions on $P$ defined by 
\eqref{formula-v} and \eqref{formula-w} respectively.
To apply Corollary \ref{ext-to-g(v,w)} to obtain an extremal metric
in the compatible K\"ahler class $[\omega_Y]$, 
we study the existence of a $g(v,w)$-soliton on $X$.
\begin{corollary}[cf. Theorem 1.1 in \cite{DJ23}]\label{2nd-main-thm''}
    Assume 
    that
    \begin{equation}\label{DJ-ineqality}
        \dim_{\mathbb C}Y+1-l_{\mathrm{ext}}(x)+\sum_{a=1}^k
    \frac{s_a-n_ac_a}{p_a(x)+c_a} > 0
    \end{equation}
    for any $x\in P$,
    where $l_{\rm ext}$ is the extremal affine function for $Y$.
    Then there exists a $g(v,w)$-soliton in $2\pi c_1(X)$.
    Consequently, there exists an extremal metric in $[\omega_Y]$.
\end{corollary}
\begin{proof}
    It follows from the formula \eqref{fiber-g(v,w)} that
    the weight $g(v,w)$ is strictly positive on $P$. 
    The condition \eqref{(v,w)_ext'} is equivalent to that
    the barycenter of $g(v,w)(x)dx/|P|$ on $P$ is $0$.
    Therefore, there exists a $g(v,w)$-soliton on $X$ 
    by analysis on toric complex Monge-Amp\`ere equation \cite{BB13}.
\end{proof}

    The condition \eqref{DJ-ineqality} corresponds to the strict positivity
    of $g(v,w)$ on $P$.
    Delcroix--Jubert \cite[Theorem 1.1]{DJ23} proved the existence of an extremal metric
    in $[\omega_Y]$ under the condition corresponding to
    the non-negativity of $g(v,w)$ on $P$.
    Their proof relies on toric geometry and proceeds through 
    algebraic stability and the coercivity of an appropriate energy functional 
    to solve the extremal metric equation. 
    In contrast, our approach via $g(v,w)$-solitons reduces the problem 
    to solving a certain complex Monge--Amp\`ere equation, 
    for which the strict positivity of $g(v,w)$ is essential. 
    Despite this limitation, we emphasize that
    our approach provides a geometric interpretation of the inequality \eqref{DJ-ineqality}.

\subsubsection{Non-toric fiber case}
We take as the fiber a Fano threefold that has been studied extensively in \cite{MW25, LWang25}.
Let $H\subset\mathbb P^3$ be a plane, and $C\subset H$ be a smooth cubic curve.
Fix a point $p\in \mathbb P^3\setminus H$.
Let us consider
\begin{equation}
    X:=\mathrm{Bl}_p\mathrm{Bl}_C\mathbb P^3
\end{equation}
the blowing up of $\mathbb P^3$ along $C$ and $p$.
The $3$-dimensional manifold $X$ is a smooth Fano manifold,
and is listed in \cite{ACC+23} as \textnumero$3.14$.
The $\mathbb C^*$-action on $\mathbb P^3$ connecting $H$ and $p$
lifts to $X$, and it is a maximal torus in $\mathrm{Aut}(X)$.
According to \cite{MW25, LWang25},
the moment polytope of the $S^1$-action
is $P:=[-1,1]$, and the Duistermaat–Heckman measure on $P$ is given by
$W(x)dx$, where
\begin{equation}
  W(x)=
  \begin{cases*}
    \frac{3}{32}(3+x)^2 & if $-1\leq x \leq0$, \\
    \frac{3}{32}(3-2x)^2 & if $0\leq x\leq 1$.
  \end{cases*}
\end{equation}
Note that our convention for the moment coordinates differs from
that of \cite{LWang25, MW25} by a factor of $-1$.
Although $X$ is not toric (more generally, not spherical), 
it satisfies {\it the weight-insensitivity} property in the sense of Delcroix \cite{De24}.
It means that for any strictly positive continuous function $g$ on the moment polytope,
the vanishing of the $g$-weighted Futaki invariant $\mathrm{Fut}_g$ of $X$
is equivalent to the $g$-weighted K-polystability of $X$ \cite{LWang25}.
In particular, there exists a $g$-soliton in $2\pi c_1(X)$ if
$\mathrm{Fut}_g$ vanishes. 
The Fano manifold $X$ does not admit a K\"ahler-Einstein metric,
and admits a Mabuchi soliton and an extremal metric
by Remark \ref{Mab_3.14}.

Let $B$ be a compact Riemann surface of genus $2$ with a K\"ahler metric
$\omega_B$ satisfying $S(\omega_B)=-2$.
Let us consider the simple principal fibration $Y$
with fiber $X$ over $B$.
The four-dimensional manifold $Y$ is not Fano.
The fibration data are given by the affine function $p(x)=x$.
For each $c>1$, the condition
$p(x)+c>0$ for any $x\in P$
defines a compatible K\"ahler class $[\omega_{Y,c}]$ on $Y$
by the formula \eqref{omega_Y}. 
Let
\begin{equation}
    l_c(x):=a_c+b_cx
\end{equation}
denote the extremal function associated
with the compatible K\"ahler class $[\omega_{Y,c}]$.
In view of \eqref{formula-v} and \eqref{formula-w},
define the weight functions on $P=[-1,1]$ as
\begin{equation}
v_c(x):=x+c
\quad\text{and}\quad
w_c(x):=l_c(x)+\frac{2}{x+c}.
\end{equation}
Then by the formula \eqref{fiber-g(v,w)},
\begin{equation}\label{formula-g_c}
    g(v_c,w_c)(x)=(x+c)\left(5-a_c-b_cx-\frac{2+c}{x+c} \right).
\end{equation}
In order to observe the existence of an extremal metric in $[\omega_{Y,c}]$
by applying Corollary \ref{ext-to-g(v,w)}, 
we determine below the range of $c>1$
for which a $g(v_c,w_c)$-soliton exists on $X$.

\begin{theorem}\label{2nd-main-thm}
    There exists a $g(v_c,w_c)$-soliton on $X$
    if and only if $c>c_{\rm sol}\approx 1.803$,
    where $c_{\rm sol}$ is the unique root greater than $1$
    of the equation $3200c^3+1033c^2-8466c-6846=0$.
    Consequently, there exists an extremal metric in $[\omega_{Y,c}]$
    if $c>c_{\rm sol}$.
    Moreover, there exists a CSCK metric in $[\omega_{Y,c}]$
    if $c=2+\frac{\sqrt{482}}{8}\approx 4.744$.
\end{theorem}

\begin{proof}
Put $g_c(x):=g(v_c,w_c)(x)$ for simplicity.
Define $I_i:=\int_{-1}^1x^iW(x)\,dx$.
Then
\begin{equation}\label{formula-I_i}
    (I_0,I_1,I_2,I_3)=\left(1, -\frac{15}{128},
    \frac{15}{64},-\frac{21}{320}\right).
\end{equation}
The condition \eqref{(v,w)_ext'} characterizing the extremal affine function is equivalent to
\begin{equation}
   -l(0)+\frac{1}{|P|}\int_Pl(x)g_c(x)W(x)\,dx=0
\end{equation}
for any affine function $l$ on $P=[-1,1]$,
where 
\begin{equation}
    |P|=\int_Pg_c(x)W(x)\,dx=\int_Pv_c(x)W(x)\,dx=cI_0+I_1.
\end{equation}
Taking $l=1$ and $l=x$, respectively, gives
\begin{equation}
\begin{cases}
    (cI_0+I_1)a_c+(cI_1+I_2)b_c=(3c-2)I_0+4I_1 \\
    (cI_1+I_2)a_c+(cI_2+I_3)b_c=(4c-2)I_1+5I_2.
\end{cases}
\end{equation}
Solving these equations, we obtain
\begin{equation}
    a_c=\frac{17700c^2-13676c-4576}{D_c}
    \quad\text{and}\quad
    b_c=\frac{-3200c^2+12800c+11300}{D_c},
\end{equation}
where $D_c=6025c^2-1042c-1290$.
Note that $D_c>0$ for any $c>1$ since 
$D_c=6025(c-1)^2+11008(c-1)+3693$.
It follows that
\begin{equation}
    g_c(x)=\frac{A_c+B_cx+C_cx^2}{D_c},
\end{equation}
where
\begin{equation}
    \begin{aligned}
    A_c &= 6400c^3-2542c^2+1500c+2580, \\ 
    B_c &= 3200c^3-375c^2-2834c-1874, \\
    C_c &= 3200c^2-12800c-11300.
    \end{aligned}
\end{equation}
If Claim \ref{claim g_c} below 
holds, 
the weight-insensitivity of $X$ implies
that a $g_c$-soliton exists on $X$ 
if and only if $c>c_{\rm sol}$.
Consequently, by Corollary \ref{ext-to-g(v,w)},
there exists an extremal metric in $[\omega_{Y,c}]$
if $c>c_{\rm sol}$.

\begin{claim}\label{claim g_c}
    $g_c(x)>0$ for any $x\in[-1,1]$
    if and only if $c>c_{\rm sol}$.
\end{claim}

To prove Claim \ref{claim g_c}, 
we first assume that $g_c(x)>0$ for all $x\in[-1,1]$.
Since
\begin{equation}
    g_c(-1)=\frac{3200c^3+1033c^2-8466c-6846}{D_c},
\end{equation}
the positivity of $g_c(-1)$ implies $c>c_{\rm sol}$
by the definition of $c_{\rm sol}$.
Conversely, assume that $c>c_{\rm sol}\approx 1.803$.
Now we have
\begin{equation}
    g_c(x)-g_c(-1)=\frac{1+x}{2D_c}
    \bigl( (1+x)B_c+(1-x)(B_c-2C_c) \bigr).
\end{equation}
For any $t\geq 0$,
\begin{equation}
    \begin{aligned}
        B_{\frac32+t} &= \frac{15325}{4}+17641t+14025t^2+3200t^3>0,\\
        B_{1+t}-2C_{1+t} &= 39917+18816t+2825t^2+3200t^3>0.
    \end{aligned}
\end{equation}
Thus $\min_{[-1,1]}g_c=g_c(-1)>0$.
This completes the proof of Claim \ref{claim g_c}.



Finally, we determine the values of $c>c_{\rm sol}$
for which a CSCK metric exists in $[\omega_{Y,c}]$.
An extremal metric has a constant scalar curvature
if and only if the extremal affine function is constant.
Let $c_*=2+\frac{\sqrt{482}}{8}$ be the unique solution
greater than $1$ of the equation $b_c=0$.
There exists a CSCK metric in $[\omega_{Y,c_*}]$.
\end{proof}

\begin{remark}\label{Mab_3.14}
    The Fano $3$-fold $X=\mathrm{Bl}_p\mathrm{Bl}_C\mathbb P^3$ used
    in Theorem \ref{2nd-main-thm} admits a Mabuchi soliton
    and an extremal metric in $2\pi c_1(X)$.
    To sse this, 
    let $l(x)$ be the unique affine function on
    the moment polytope $[-1,1]$ satisfying
    \begin{equation}
        \int_{-1}^1l(x)W(x)\,dx=1
        \quad\text{and}\quad
        \int_{-1}^1xl(x)W(x)\,dx=0.
    \end{equation}
    By \eqref{formula-I_i}, we have $l(x)=\frac{128}{241}(x+2)$
    which is positive on $[-1,1]$.
    Thus, by the weight-insensitivity,
    $X$ admits an $l$-soliton which is nothing but a Mabuchi soliton.
    By Theorem \ref{meta-theorem}, $X$ also admits an extremal metric.
\end{remark}

\begin{remark}
    Let $H\subset\mathbb{P}^3$ be a plane and $C\subset H$ be a smooth
    cubic curve. Define $X^{\prime}:=\mathrm{Bl}_C\mathbb P^3$.
    The $3$-fold $X^{\prime}$ is a smooth Fano manifold listed in
    \cite{ACC+23} as \textnumero$2.28$.
    According to \cite{MW25, LWang25},
    a maximal torus of $\mathrm{Aut}(X^{\prime})$
    is $\mathbb{C}^*$, the moment polytope is $[-3,1]$, and
    the Duistermaat–Heckman measure is given by
    \begin{equation}
  \begin{cases*}
    \frac{3}{40}(3+x)^2\,dx & if $-3\leq x \leq0$, \\
    \frac{3}{40}(3-2x)^2\,dx & if $0\leq x\leq 1$.
  \end{cases*}
\end{equation}
By \cite{LWang25}, $X^{\prime}$ is weight-insensitive.
The Fano $3$-fold $X^\prime$ does not admit a Mabuchi soliton, 
since the unique affine function obtained by an argument analogous 
to that in Remark \ref{Mab_3.14} changes sign on the moment polytope.

Using a similar argument as in the proof of Theorem \ref{2nd-main-thm},
we can show that the weight $g(v_c,w_c)$,
which is given by the same formula as \eqref{formula-g_c}, 
changes sign on $[-3,1]$ for any $c>3$.
Consequently, $X^{\prime}$ does not admit a $g(v_c,w_c)$-soliton
for any $c>3$.
Thus, Corollary \ref{ext-to-g(v,w)} does not produce an extremal metric
on the simple principal fibration $Y^{\prime}$ with fiber $X^{\prime}$
over a compact Riemann surface of genus $2$.
This does not imply that $Y'$ does not admit an extremal metric
in compatible K\"ahler classes.
\end{remark}
\section{$(v,w)$-CSCK metrics and $g(v,w)$-solitons in geometry of Sasaki manifolds}
\label{Sasaki-section}

In this section, we apply Theorem \ref{main-theorem} to transverse
K\"ahler geometry on Sasaki manifolds. We consider normalized Reeb
fields on the unit circle bundle of the canonical bundle of a Fano
manifold. By comparing the corresponding spaces of transverse
K\"ahler potentials, we express the transverse $g$-extremal and
$g$-soliton equations as weighted equations on the regular Fano
quotient. The resulting weights satisfy precisely the relation
defining $g_{(v,w)}$, allowing us to establish an equivalence between
the corresponding existence problems. For positive affine weights,
this yields a Sasakian analog of Theorem \ref{meta-theorem}, relating
Sasaki extremal metrics to Sasaki Mabuchi solitons. We conclude by
determining exactly which normalized Reeb fields on the standard
odd-dimensional sphere admit Sasaki Mabuchi solitons.


We refer to \cite{BGBook, MSY08} for basics of Sasaki manifolds.
We also refer to \cite{FOW09, CL21} for differential geometry in Sasaki manifolds
including transverse K\"ahler geometry.

\subsection{Basic setup}
Let $(N,g)$ be a compact Sasaki manifold of real dimension $2n+1$ for $n\geq 1$ i.e.
the cone $(C(N), \bar{g}):=(N\times\mathbb{R}_{+}, r^{2}g+dr^{2})$ is a K\"ahler manifold
of complex dimension $n+1$.
The manifold $N$ is often identified with the submanifold $\{r=1\}\subset C(N)$.
A Sasaki manifold $(N,g)$ admits a contact structure $(\widehat{\xi}, \eta,\Phi)$
and also admits a one dimensional foliation $\mathcal{F}_{\widehat{\xi}}$.
Here the Killing vector field $\widehat{\xi}$ on $(N,g)$ is called the Reeb vector field,
$\eta$ is called the contact $1$-form on $N$,
$\Phi$ is a $(1,1)$ tensor on $N$ which defines a complex structure on the contact sub-bundle
$D:=\Ker\eta$.
A Sasaki manifold is also denoted by $(N,\widehat{\xi},\eta,D,\Phi)$.
The data $(\widehat{\xi},\eta,D,\Phi)$ is called a Sasaki structure on $N$.
A Sasaki manifold $(N,\widehat{\xi},\eta,D,\Phi)$ is said to be
(i) regular if $\mathcal{F}_{\widehat{\xi}}$ is obtained by a free $S^{1}$-action,
(ii) quasi-regular if all the leaves of
$\mathcal{F}_{\widehat{\xi}}$ are compact
and (iii) irregular if it is not quasi-regular.

The Sasaki structure $(\widehat{\xi},\eta,D,\Phi)$ on $N$
yields a transverse holomorphic structure
and a transverse K\"ahler structure on $\mathcal{F}_{\widehat{\xi}}$.
One has the splitting of
complex valued basic $p$-forms and the derivative:
\begin{equation}
\Lambda_{B}(N)^{p}=\oplus_{i+j=p}\Lambda_{B}^{i,j}(N),
\quad d_{B}:=d|_{\Lambda_{B}^{p}(N)}=\partial_{B}+\dbar_{B}.
\end{equation}
Put $d_{B}^{c}=\frac{\sqrt{-1}}{2}(\dbar_{B}-\partial_{B})$ so that
$d_{B}d_{B}^{c}=\deldelb$.
If we emphasize the Reeb field $\widehat{\xi}$, we write
$d_{\widehat{\xi}}, \partial_{\widehat{\xi}}, \dbar_{\widehat{\xi}}$ and $d^{c}_{\widehat{\xi}}$.
Let $\omega$ 
be a basic K\"ahler form on $N$.
The transverse Ricci form $\Ric^{T}$ for $\omega$ is defined as
$\Ric^{T}=-\deldelb\log\det\omega^{n}$.
The basic cohomology class represented by $\Ric^{T}/2\pi$ is independent of the choice
of $\omega$.
This basic cohomology class $[\Ric^{T}/2\pi]_{B,\widehat{\xi}}$ is called the basic first Chern class
and is denoted by $c^{B,\widehat{\xi}}_{1}(N)$.
\subsection{Transverse $g$-solitons and $g$-extremal metrics}
Let $\Ric^{T}$ be the transverse Ricci form for
$\omega^{T}:=d\eta$.
The transverse scalar curvature $S^{T}$ for $\omega^{T}$ is defined by
$S^{T}(\omega^{T})^{n}\wedge\eta
=n\Ric^{T}\wedge(\omega^{T})^{n-1}\wedge\eta$.
Let $V:=\int_{N}(\omega^{T})^{n}\wedge\eta$ be the volume of $N$
which is independent of the choice of the transverse K\"ahler metric in
$[d\eta]_{B,\widehat{\xi}}$.
The average
$\underline{S}^{T}=\int_{N}S^{T}(\omega^{T})^{n}\wedge\eta/V$
of $S^{T}$ is also independent of the choice of $\eta$.

Let $\widehat{T}$ be a real maximal compact torus acting effectively as CR-automorphisms of
$(D=\ker\eta, \Phi|_{D})$ such that
$\widehat{\xi}\in\widehat{\frak{t}}:=\mathrm{Lie}(\widehat{T})$.
One has the contact moment map
$\mu^{\widehat{\xi}}:C(N)\to\widehat{\mathfrak{t}}^{*}$ defined by
\begin{equation}
\langle\mu^{\widehat{\xi}},u\rangle=r^{2}\eta(u).
\end{equation}
Let $\mathcal{C}^{*}\subset\widehat{\frak{t}^{*}}$ be
the image of $\mu^{\widehat{\xi}}$
which is a convex rational polyhedral cone.
The image of the restriction $\mu^{\widehat{\xi}}|_{N}$ is given by
\begin{equation}
P_{\widehat{\xi}}:=\Set{x\in\mathcal{C}^{*} | \langle x,\widehat{\xi}\rangle=1}
\end{equation}
and is called the characteristic hyperplane.
Let $\mathcal{C}:=\Set{y\in\widehat{\frak{t}} | \langle x,y\rangle >0 \text{ for any } x\in\mathcal{C}^{*}}$
be the dual cone of $\mathcal{C}^{*}$.
An element $\widehat{\chi}\in\widehat{\frak{t}}$ lies in $\mathcal{C}$
if and only if $\eta(\widehat{\chi})>0$.
Thus $\widehat{\chi}\in\mathcal{C}$ defines a Reeb field on $N$  \cite{HS16}.
The cone $\mathcal{C}$ is called the Reeb cone or the Sasaki cone.
For any Reeb field $\widehat{\chi}\in\mathcal{C}$,
there exists a natural projection
\begin{equation}\label{projection}
P_{\widehat{\xi}}\ni x\mapsto \frac{x}{\langle x,\widehat{\chi}\rangle} \in P_{\widehat{\chi}}.
\end{equation}

In order to define the notions of a transverse $g$-extremal metric
and a transverse $g$-soliton,
let $g$ be a smooth positive function on $P_{\widehat{\xi}}$
satisfying 
\begin{equation}\label{norm-g}
\int_{N}g(\mu^{\widehat{\xi}})(\omega^{T})^{n}\wedge\eta=V.
\end{equation}
The positivity of the weight function 
$g$ is not necessarily required to define transverse
$g$-extremal metrics, whereas it is essential for defining transverse 
$g$-solitons.
\begin{definition}
A transverse K\"ahler metric $\omega^{T}$ is called 
{\it a transverse
$g$-extremal metric} if
\begin{equation}
S^{T}-\underline{S}^{T}=1-g(\mu^{\widehat{\xi}}|_{N}).
\end{equation}
If $g$ is an affine function, the transverse $g$-extremal metric
is {\it an extremal Sasaki metric} in the sense of
Boyer-Galicki-Simanca \cite{BGS08}.
\end{definition}

An affine function appearing in the definition of an extremal Sasaki metric is uniquely determined, regardless of whether such a metric exists.
Let $\Aff(P_{\widehat{\xi}})$ be the set of affine functions on $P_{\widehat{\xi}}$.
For any $l\in\Aff(P_{\widehat{\xi}})$, the value
\begin{equation}
    \mathrm{Fut}_{\widehat{\xi}}(l)
    :=\frac{1}{V}\int_N l(\mu^{\widehat{\xi}})
    \left(S^T-\underline{S}^T\right)(\omega^T)^n\wedge\eta
\end{equation}
is invariant under the deformation
$\eta\mapsto\eta+d^c_B\phi$ for any basic function $\phi$ on $N$.
The function $\mathrm{Fut}_{\widehat{\xi}}$ is called {\it the Sasaki Futaki invariant}.
The non-vanishing of $\mathrm{Fut}_{\widehat{\xi}}$ obstructs the existence of
a $\widehat{\xi}$-transverse CSCK metric in $[d\eta]_{B,\widehat{\xi}}.$
The Futaki-Mabuchi pairing for $l,l^\prime\in\Aff(P_{\widehat{\xi}})$ is defined by
\begin{equation}
    \langle l,l^\prime\rangle_{\widehat{\xi}}:=\frac{1}{V}
    \int_Nl(\mu^{\widehat{\xi}})l^\prime(\mu^{\widehat{\xi}})
    (\omega^T)^n\wedge\eta,
\end{equation}
which is also invariant under the deformation $\eta\mapsto\eta+d^c_B\phi$.
Let $l_{\widehat{\xi},\rm ext}\in\Aff(P_{\widehat{\xi}})$ be a uniquely determined
affine function satisfying
\begin{equation}
    \mathrm{Fut}_{\widehat{\xi}}(l)=
    \langle l, l_{\widehat{\xi},\rm ext}-\underline{S}^T\rangle_{\widehat{\xi}} 
\end{equation}
for any $l\in\Aff(P_{\widehat{\xi}})$.
The affine function $l_{\widehat{\xi},\rm ext}$ is the dual of the Sasaki Futaki invariant,
and is called {\it the extremal affine function}.
Thus an extremal Sasaki metric $\omega^T$ must satisfy
$S^T=l_{\widehat{\xi},\rm ext}(\mu^{\widehat{\xi}})$.

Let us now assume transversely Fano condition
$2\pi c^{B,\widehat{\xi}}_{1}(N)=[d\eta]_{B,\widehat{\xi}}$.
In this case $\underline{S}^{T}=n$.
The condition $2\pi c^{B,\widehat{\xi}}_{1}(N)=[d\eta]_{B,\widehat{\xi}}$ follows from
the topological assumptions
$c^{B,\widehat{\xi}}_{1}(N)>0$ and $c_{1}(D)=0$
and so-called the $D$-homothetic transformation for the radial function on
the Riemannian cone $C(N)$.
In fact $c_{1}(D)=0$ if and only if $2\pi c^{B,\widehat{\xi}}_{1}(N)$ is represented by
$\tau d\eta$ for some constant $\tau$ \cite{FOW09}.

\begin{definition}
A transverse K\"ahler metric $\omega^{T}$ is called
{\it a transverse $g$-soliton} if
\begin{equation}
\Ric^{T}-\omega^{T}=\deldelb \log g(\mu^{\widehat{\xi}}|_{N})
\end{equation}
If $g=e^{l}$ for an affine function $l$, the transverse $g$-soliton
is {\it a Sasaki Ricci soliton} in the sense of Futaki-Ono-Wang \cite{FOW09}.
If $g$ is an affine function, we call the transverse $g$-soliton
{\it a Sasaki Mabuchi soliton}.
\end{definition}

A Sasaki Mabuchi soliton is expressed by $l_{\widehat{\xi},\rm ext}$.
By the transverse Fano condition,
there exists a basic function $\rho^T$, called the transverse Ricci potential,
satisfying
\begin{equation}
    \Ric^T-\omega^T=\deldelb\rho^T
    \quad\text{and}\quad
    \int_N(1-e^{\rho^T})(\omega^T)^n\wedge\eta=0.
\end{equation}
Similar to the K\"ahler case, one has the formula
\begin{equation}
    \mathrm{Fut}_{\widehat{\xi}}(l)
    =\frac{1}{V}\int_Nl(\mu^{\widehat{\xi}})\left(1-e^{\rho^T}\right)
    (\omega^T)^n\wedge\eta.
\end{equation}
Thus a Sasaki Mabuchi soliton $\omega^T$ must satisfy
$1-e^{\rho^T}=l_{\widehat{\xi},\rm ext}(\mu^{\widehat{\xi}})-n$.
Define the Sasaki Mabuchi constant as
\begin{equation}
    M_{N,\widehat{\xi}}:=\sup_{P_{\widehat{\xi}}}
    \left(l_{\widehat{\xi},\rm ext}-n\right).
\end{equation}
The condition $M_{N,\widehat{\xi}}<1$ 
is necessary for
the existence
of a Sasaki Mabuchi soliton in $2\pi c^{B,\widehat{\xi}}_{1}(N)$
as an analog of the observation by Mabuchi \cite{Ma01}.
\subsection{Regular quotient}\label{regular-quotient}

Let $X$ be an $n$-dimensional Fano manifold, let $T$ be a maximal compact
torus in $\operatorname{Aut}(X)$, and fix a $T$-invariant K\"ahler form
$\omega_0\in2\pi c_1(X)$. We use the canonically normalized moment maps
$m_\varphi:X\to P_X\subset\mathfrak t^*$
as in \eqref{can-lift}, where
$\mathfrak t=\operatorname{Lie}(T)$
 and
 $\omega_\varphi=\omega_0+\deldel\phi$. 

Let $I$ be a complex structure on $K_X^\times$, and
let $\pi:K_X^\times\to X$ be the bundle projection. Choose a $T$-invariant Hermitian
metric on $K_X$ with curvature $-\omega_0$, and let
$N_{\omega_0}\subset K_X^\times$ be its unit circle bundle, with
$\pi_0=\pi|_{N_{\omega_0}}$.
Let $\widehat{\chi}$ be the generator of the fiberwise $S^1$-action, normalized to
have period $2\pi$. Denote by $\eta_0^{\widehat{\chi}}$ the connection form induced
by the Chern connection of the chosen Hermitian metric, satisfying
\begin{equation}
 \eta_0^{\widehat{\chi}}(\widehat{\chi})=1
 \quad\text{and}\quad
 d\eta_0^{\widehat{\chi}}=\pi^*\omega_0.
\end{equation}
Its kernel $D_0$ has the induced CR structure $J^{\widehat{\chi}}$. In the convention
$\omega^T=d\eta$ used in this section,
we have the transversely Fano condition
 $2\pi c_1^{B,\widehat{\chi}}(N_{\omega_0})
 =[d\eta_0^{\widehat{\chi}}]_{B,\widehat{\chi}}$.


\subsubsection{Normalization of the Reeb field}
The decomposition $\widehat{\frak t}=\frak t\oplus\mathbb R\widehat{\chi}$ depends on the
\emph{natural} lift of the action. For $\tilde{\zeta}=\zeta+a\widehat{\chi}\in\widehat{\frak t}$, set
\begin{equation}
l_{\zeta}(x)=\langle x,\zeta\rangle+a.
\end{equation}
Then $\tilde{\zeta}$ belongs to the Reeb cone if $l_{\zeta}>0$ on $P_X$.
Then its contact form for the reference CR structure is
$\eta_0^{\tilde{\zeta}}=\eta_0^{\widehat{\chi}}/l_{\zeta}(\pi^*m_0)$.
Moreover,
 $2\pi c_1^{B,\tilde{\zeta}}(N_{\omega_0})=a[d\eta_0^{\tilde{\zeta}}]_{B,\tilde{\zeta}}.$
Indeed, the canonical holomorphic volume form on $K_X^\times$ is invariant
under the natural $T$-action and has weight $1$ under fiberwise $S^1$-action
(see \cite[Section 4.3]{ALL25} for example). 
Since $0$ lies in the
interior of the canonically normalized polytope $P_X$, positivity of
$l_{\zeta}$ implies $a>0$. 
Thus each ray $\tilde{\zeta}$ in the Reeb cone has a unique representative
with coefficient $a=1$.

Hereafter, we fix a Reeb field 
\begin{equation}\label{Reeb-norm}
\widehat{\xi}=\xi+\widehat{\chi}
\end{equation}
for $\xi\in\frak{t}$
satisfying $l_\xi(x)=\langle x,\xi\rangle+1>0$ for any $x\in P_X$.
In particular,
$2\pi c_1^{B,\widehat{\xi}}(N_{\omega_0})=[d\eta_0^{\widehat{\xi}}]_{B,\widehat{\xi}}$.

\subsubsection{Transverse potentials and their identification}
For $b\in\{\widehat\chi,\widehat\xi\}$, let
$Q_b=TN_{\omega_0}/\mathbb Rb$ carry the reference transverse complex
structure, and let $d_b^c$ be its basic $d^c$-operator. In particular,
$d_b^cu=-\frac12du\circ J^b$, where $J^b$ extends the reference
CR complex structure by $J^b b=0$. 
Put $\eta_u^b=\eta_0^b+d_b^cu$.
Following \cite[Section 2.3]{ACL21}
 (building on an idea in \cite[Proof of Lemma 2.2]{HS16}),
we introduce the space of $b$-transverse K\"ahler potentials
\begin{equation} \label{potential-space}
 \Xi_b^{\widehat T}
 =\Set{u\in C^\infty(N_{\omega_0},\mathbb R)^{\widehat T} |
       d\eta_u^b>0\ \text{on }Q_b}.
\end{equation}
The CR structure on $\ker\eta_u^b$ is induced from this fixed
transverse complex structure.
Note that
$\Xi_{\widehat\chi}^{\widehat T}=\pi_0^*\mathcal H^T$.
We often identify an element $u\in\Xi_b^{\widehat{T}}$ with
the corresponding $b$-transverse K\"ahler metric $d\eta_u^b$.

To compare the two potential spaces, let $E_b=-Ib$ and let
$p_b:K_X^\times\to N_{\omega_0}$ be radial projection along the flow of $E_b$.
Fix the reference radial functions defined by
\begin{equation}
 E_b r_{b,0}=r_{b,0}
 \quad\text{and}\quad
 r_{b,0}|_{N_{\omega_0}}=1.
\end{equation}
The radial function associated to
$u\in\Xi_b^{\widehat T}$ is
\begin{equation}
 r_{b,u}=e^{p_b^*u/2}r_{b,0},
 \end{equation}
 which satisfies
 $2d^c\log r_{b,u}=p_b^*\eta_u^b.$
According to
\cite[Section 2.3, Corollary 2.15]{ACL21},
there exists a natural bijection
\begin{equation}
 \Theta_{\widehat\chi,\widehat\xi}:
 \Xi_{\widehat\xi}^{\widehat T}
 \to\Xi_{\widehat\chi}^{\widehat T}
 =\pi_0^*\mathcal H^T.
\end{equation}
Geometrically, it is characterized by the condition that
if $\Theta_{\widehat\chi,\widehat\xi}(\widehat\varphi)=\pi_0^*\varphi$,
then
\begin{equation}
 \Sigma
 :=\{r_{\widehat\xi,\widehat\varphi}=1\}
 =\{r_{\widehat\chi,\pi_0^*\varphi}=1\}\subset K_X^\times.
\end{equation}

Fix a corresponding pair of potentials
$\pi_0^*\varphi\in\Xi_{\widehat\chi}^{\widehat T}$ and 
$\widehat\varphi\in\Xi_{\widehat\xi}^{\widehat T}$,
and work on the hypersurface $\Sigma$. Define
\begin{equation}
 \alpha:=\left.2d^c\log r_{\widehat\chi,\pi_0^*\varphi}\right|_{T\Sigma}
 \quad\text{and}\quad
 \beta:=\left.2d^c\log r_{\widehat\xi,\widehat\varphi}\right|_{T\Sigma}.
\end{equation}
Both contact forms induce the CR structure
 $\mathcal D_\Sigma:=T\Sigma\cap IT\Sigma$
 and $J_\Sigma:=I|_{\mathcal D_\Sigma}.$
Their Reeb fields are $\widehat\chi$ and $\widehat\xi$, respectively.
The maps $p_b|_\Sigma:\Sigma\to N_{\omega_0}$ are $\widehat T$-equivariant CR
diffeomorphisms to the corresponding deformed structures on $N_{\omega_0}$,
and satisfy
\begin{equation}\label{alpha-beta}
 \alpha=(p_{\widehat\chi}|_\Sigma)^*\eta_{\pi_0^*\varphi}^{\widehat\chi}
 \quad\text{and}\quad
 \beta=(p_{\widehat\xi}|_\Sigma)^*\eta_{\widehat\varphi}^{\widehat\xi}.
\end{equation}
Since $p_{\widehat\xi}|_\Sigma$ is a $\widehat{T}$-equivariant CR diffeomorphism, 
the transverse $g$-soliton and $g$-extremal equations for $\eta_{\widehat\varphi}^{\widehat\xi}$
on the fixed $S^1$-bundle $N_{\omega_0}$ are equivalent to
the corresponding equations for $\beta$ on the common hypersurface $\Sigma$.
The same identification applies to the $\widehat\chi$-transverse structure
via $p_{\widehat\chi}|_\Sigma$.

Let $\pi_\Sigma:=\pi|_\Sigma=\pi_0\circ p_{\widehat\chi}|_\Sigma$.
This is the regular quotient by $\widehat\chi$.
Since $\Ker\alpha=\Ker\beta$, we obtain
\begin{equation}\label{eta-equation}
 d\alpha=\pi_\Sigma^*\omega_\varphi,
 \quad\text{and}\quad
 \beta=\frac{\alpha}{l_\xi(m_\varphi)\circ\pi_\Sigma}
\end{equation}
All contact moment maps and transverse curvature quantities below are
evaluated on $\Sigma$. 
Its CR structure is fixed while the two Reeb fields 
$\widehat{\chi}$ and $\widehat{\xi}$ are compared.

\subsubsection{The weight on the regular quotient}
Let $\widehat{\chi}^*\in\widehat{\frak t}^*$ be characterized by
$\widehat{\chi}^*|_{\frak t}=0$ and $\langle\widehat{\chi}^*,\widehat{\chi}\rangle=1$.
We regard $\frak t^*$ as the annihilator of $\widehat{\chi}$.
Write $\mu_\alpha,\mu_\beta:\Sigma\to\widehat{\mathfrak t}^*$ for
the contact moment maps, so that
$\langle\mu_\alpha,\widehat{\zeta}\rangle=\alpha(\widehat{\zeta})$ and
$\langle\mu_\beta,\widehat{\zeta}\rangle=\beta(\widehat{\zeta})$
for $\widehat{\zeta}\in\widehat{\frak t}$.
By equation \eqref{eta-equation}, we have
\begin{equation}\label{cont-identity2}
\mu_\alpha=\pi_\Sigma^*m_\varphi+\widehat\chi^*
\quad\text{and}\quad
 \mu_\beta
 =\frac{\pi_\Sigma^*m_\varphi+\widehat\chi^*}
        {\ell_\xi(m_\varphi)\circ\pi_\Sigma}.
\end{equation}

Identify the characteristic hyperplane for $\widehat{\chi}$ with
$P_X+\widehat{\chi}^*$. 
The radial projection between the two characteristic hyperplanes
of the moment cone defines a map $q_\xi:P_X\to P_{\widehat{\xi}}$
given by
\begin{equation}
 q_\xi(x)=\frac{x+\widehat{\chi}^*}{l_\xi(x)}.
\end{equation}
This map is the restriction of a diffeomorphism from
$\{x\in\frak t^* : l_\xi(x)>0\}$ to
$\{y\in\widehat{\frak t}^*: \langle y,\widehat{\xi}\rangle=1,
\ \langle y,\widehat{\chi}\rangle>0\}$, with inverse
\begin{equation}\label{q^{-1}}
 q_\xi^{-1}(y)=
 \frac{\mathrm{pr}_{\frak t^*}y}{\langle y,\widehat{\chi}\rangle},
\end{equation}
where $\mathrm{pr}_{\frak t^*}:\widehat{\frak t}^*\to\frak t^*$ is the projection.
For a fixed smooth positive function $g$ on $P_{\widehat{\xi}}$, define
\begin{equation}
  g_0:=g\circ q_\xi
\end{equation}
By the formulas \eqref{cont-identity2} and \eqref{q^{-1}},
\begin{equation}
    g\bigl(\mu_{\beta}\bigr)
 =\pi_{\Sigma}^*\bigl(g_0\circ m_\phi\bigr).
\end{equation}



\begin{proposition}[See Corollary 4.1 in \cite{ALL25} for a related result]\label{reg-sigma}
Let $\phi\in\mathcal H^T$ and let $\widehat{\phi}\in\Xi_{\widehat\xi}^{\widehat T}$
satisfy $\Theta_{\widehat\chi,\widehat\xi}(\widehat{\phi})=\pi_0^*\phi$.
The potential $\widehat\phi$ defines a $\widehat{\xi}$-transverse $g$-soliton
if and only if the K\"ahler metric $\omega_\phi\in 2\pi c_1(X)$
is a $g_\xi$-soliton on $X$, where
\begin{equation}
 g_\xi(x)=l_\xi(x)^{-(n+2)}g_0(x).
\end{equation}
\end{proposition}

\begin{proof}
Let $\alpha$ and $\beta$ be the contact forms given by the formula \eqref{alpha-beta}.
By \cite[Lemma 4.3]{ALL25},
the defining equation for a $\widehat{\xi}$-transverse $g$-soliton
\begin{equation}
    \Ric^T_{\widehat{\xi}}(d\beta)-d\beta=dd^c_{\widehat{\xi},\beta}\log g(\mu_\b)
\end{equation}
is equivalent to
\begin{equation}
    \Ric^T_{\widehat{\chi}}(d\a)-d\a
    =dd^c_{\widehat{\chi},\a}
    \left(\log \pi_{\Sigma}^*g_0(m_\phi)-(n+2)\log \pi_{\Sigma}^*l_{\xi}(m_\phi)\right).
\end{equation}
The last equation is nothing but
$\Ric(\omega_\phi)-\omega_\phi=dd^c\log g_\xi(m_\phi)$.
This completes the proof.
\end{proof}

\begin{proposition}\label{reg-ext}
Let $\phi\in\mathcal H^T$ and let $\widehat{\phi}\in\Xi_{\widehat\xi}^{\widehat T}$
satisfy $\Theta_{\widehat\chi,\widehat\xi}(\widehat{\phi})=\pi_0^*\phi$.
The potential $\widehat\phi$ defines a $\widehat{\xi}$-transverse $g$-extremal metric
if and only if the K\"ahler metric $\omega_\phi\in 2\pi c_1(X)$
is a $(v_\xi,w_\xi)$-CSCK metric on $X$, where
\begin{equation}
 v_\xi= l_\xi^{-(n+1)}
 \quad\text{and}\quad
 w_\xi=\frac{n+1-g_0}{l_\xi}.
\end{equation}
\end{proposition}

\begin{proof}
Let $\alpha$ and $\beta$ be the contact forms given by the formula \eqref{alpha-beta}.
By \cite[Lemma 3]{AC21},
the $\widehat{\xi}$-transverse scalar curvature of $d\beta$
is expressed as
\begin{equation}
 S^T_{\widehat{\xi}}(\widehat\phi)
 =\pi_{\Sigma}^*\left( l_\xi(m_\phi)S_{v_\xi}(\phi) \right).
\end{equation}
Since $\overline S^T_{\widehat{\xi}}=n$,
the defining equation for a $\widehat{\xi}$-transverse $g$-extremal metric becomes
$S^T_{\widehat{\xi}}(\widehat{\phi})=n+1-g(\mu_\b)$.
This equation is nothing but
$S_{v_\xi}(\phi)=w_\xi(m_\phi)$.
This completes the proof.
\end{proof}

\begin{lemma}\label{Sasaki-g(v,w)}
Let $g_\xi$ be the weight on $P_X$ given in Proposition \ref{reg-sigma}.
Let $v_\xi$ and $w_\xi$ be the weights on $P_X$ given in Proposition \ref{reg-ext}.
Then we have
\begin{equation}
 g(v_\xi,w_\xi)=g_\xi.
\end{equation}
\end{lemma}

\begin{proof}
By the Reeb normalization \eqref{Reeb-norm},
the constant term in $l_\xi$ is $1$.
We have
\begin{equation}
 \langle d\log v_\xi(x),x\rangle=-(n+1)\frac{l_\xi(x)-1}{l_\xi(x)}.   
\end{equation}
Hence
\begin{align*}
 g(v_\xi,w_\xi)
 &=l_\xi^{-(n+1)}
 \left(n+1-(n+1)\frac{l_\xi-1}{l_\xi}-\frac{n+1-g_0}{l_\xi}\right)\\
 &=l_\xi^{-(n+2)}g_0.
\end{align*}
\end{proof}

As a consequence of Theorem \ref{main-theorem},
Propositions \ref{reg-sigma} and \ref{reg-ext}, and Lemma \ref{Sasaki-g(v,w)},
we obtain a Sasakian analog of a result in \cite{ALN25, HN24}.

\begin{theorem}\label{Sasaki-g-ext-sol}
Fix a Reeb field $\widehat{\xi}=\xi+\widehat{\chi}\in\widehat{\frak{t}}$
on $N_{\omega_0}$,
and fix a smooth positive weight $g$ on the characteristic hyperplane $P_{\widehat{\xi}}$ satisfying the normalization condition \eqref{norm-g}.
 The following are equivalent.
    \begin{enumerate}
        \item There exists a $\widehat{\xi}$-transverse $g$-extremal metric
        in $\Xi_{\widehat{\xi}}^{\widehat T}$.
        \item There exists a $\widehat{\xi}$-transverse $g$-soliton
        in $\Xi_{\widehat{\xi}}^{\widehat T}$.
    \end{enumerate}
\end{theorem}
When $g$ is a positive affine function, we obtain
a Sasakian analog of Theorem \ref{meta-theorem}.
\begin{corollary}\label{Sasaki-ext-sol}
Fix a Reeb field $\widehat{\xi}=\xi+\widehat{\chi}\in\widehat{\frak{t}}$
on $N_{\omega_0}$.
    The following are equivalent.
    \begin{enumerate}
        \item There exists a Sasaki extremal metric in
        $\Xi_{\widehat{\xi}}^{\widehat T}$
        whose transverse scalar curvature is less than $n+1$.
        \item There exists a Sasaki Mabuchi soliton in
        $\Xi_{\widehat{\xi}}^{\widehat T}$.
    \end{enumerate}
\end{corollary}
Note that the condition (1) in Corollary \ref{Sasaki-ext-sol}
is equivalent to that
there exists a Sasaki extremal metric in $\Xi_{\widehat{\xi}}^{\widehat T}$
and the Sasaki Mabuchi constant 
$M_{N_{\omega_0},\widehat{\xi}}=\sup_{P_{\widehat{\xi}}}
    (l_{\widehat{\xi},\rm ext}-n)$ is less than $1$.

\subsection{Sasaki Mabuchi soliton on $S^{2n+1}$}
Let us consider the standard odd-dimensional sphere $S^{2n+1}$,
which is the most fundamental example of Sasaki manifolds.
Boyer-Galicki-Simanca \cite{BGS08} showed that there exists a Sasaki extremal
metric on $S^{2n+1}$ for any Reeb field.
In this subsection, by applying Corollary \ref{Sasaki-ext-sol},
we determine Reeb fields on $S^{2n+1}$
admitting a Sasaki Mabuchi soliton.

The odd-dimensional sphere
$S^{2n+1}=\Set{(w_{0},\dots,w_{n})\in\mathbb C^{n+1} | \sum_{i=0}^{n}|w_{i}|^{2}=1}$
is a (toric) Sasaki manifold,
since the cone is given by $C(S^{2n+1})=\mathbb{C}^{n+1}\setminus\{0\}$
which has a natural action of $\widehat{T}:=(S^{1})^{n+1}$.
Put $w_{i}=u_{i}+\sqrt{-1}v_{i}$ and define 
\begin{equation}
 \xi_{i}:=-v_{i}\frac{\partial}{\partial u_{i}}
+u_{i}\frac{\partial}{\partial v_{i}}.
\end{equation}
The standard contact form and the standard Reeb field are respectively given by
\begin{equation}
 \eta_{\mathrm{std}}:=\sum_{i=0}^{n}(-v_{i}du_{i}+u_{i}dv_{i})
 \quad\text{and}\quad
\xi_{\mathrm{std}}:=
\sum_{i=0}^n\xi_{i}.
\end{equation}
The transverse K\"ahler structure with respect to $\xi_{\mathrm{std}}$
has a transverse K\"ahler-Einstein metric which comes from
the Hopf fibration $S^{2n+1}\to\mathbb{P}^{n}$. 
The Lie algebra $\widehat{\frak t}$ of $\widehat{T}$ is generated by
$\xi_{0},\xi_1\dots,\xi_n$.
Thus one gets the identification $\widehat{\frak t}\cong\mathbb R^{n+1}$.
Since the contact moment map for $\xi_{\mathrm{std}}$ is given by
$\mu^{\xi_{\mathrm{std}}}(w)=(|w_{0}|^{2},\dots,|w_{n}|^{2})$,
the moment cone $\mathcal C^*$ and the Reeb cone $\mathcal C$ are given by
$$\mathcal C^{*}=\Set{x=(x_{0},\dots,x_{n})\in\mathbb R^{n+1}\setminus\{0\} |
x_{i}\geq 0 \text{ for any } i=0,\dots,n}$$
and $\mathcal C=(\mathbb{R}_{>0})^{n+1}$.
For $a=(a_{0},\dots,a_{n})\in\mathcal C$,
define
\begin{equation}
 \xi_{a}:=\sum_{i=0}^{n}a_{i}\xi_{i}.   
\end{equation}
The Reeb field
$\xi_{a}$ is quasi-regular
if $a\in(\mathbb Q_{>0})^{n+1}$,
and $\xi_{a}$ is regular if $a_{0}=\cdots=a_{n}$.
Define the contact form for $\xi_a$ as
\begin{equation}
    \eta_a:=\frac{\eta_{\rm std}}{\sum_{i=0}^na_i|w_i|^2}.
\end{equation}
Then $2\pi c_1^{B,\xi_a}(S^{2n+1})=(\sum_{i=0}^na_i)[d\eta_a]_{B,\xi_a}$.
We introduce the space of normalized Reeb fields
\begin{equation}
    \mathcal C_1:=\Set{a=(a_0,\dots,a_n)\in\mathcal C | \sum_{i=0}^na_i=1}.
\end{equation}
For $a\in\mathcal C_1$, the characteristic hyperplane for $\xi_a$ is given by
\begin{equation}
    \mathcal P_a:=\mathcal P_{\xi_a}=
    \Set{x=(x_0,\dots,x_n)\in\mathcal C^* | \sum_{i=0}^na_ix_i=1}.
\end{equation}

\begin{theorem}\label{Mab-sol-sphere}
    Let $a\in\mathcal C_1$.
    There exists a $\widehat{T}$-invariant Sasaki Mabuchi soliton for $\xi_a$
    if and only if
    \begin{equation}
        \min_i a_i>\frac{1}{n+2}.
    \end{equation}
\end{theorem}

\begin{proof}
    We apply Corollary \ref{Sasaki-ext-sol}.
    Since the unit circle bundle of $\mathcal O(-1)$
    over $\mathbb P^n$ is $S^{2n+1}$,
    the unit circle bundle of $K_{\mathbb P^n}=\mathcal{O}(-n-1)$ is
    the finite quotient $S:=S^{2n+1}/\mathbb Z_{n+1}$.
    Here $\mathbb Z_{n+1}$ is the group of $(n+1)$-th roots of unity,
    acting freely and diagonally on $S^{2n+1}$.
    The generator of the fiber rotation on $S$, normalized to have
    period $2\pi$, lifts to $(n+1)^{-1}\sum_{i=0}^n\xi_i$ on $S^{2n+1}$.
    Since the deck group of $S^{2n+1}\to S$ is contained in $\widehat{T}$,
    invariant transverse potentials descend to $S$ and lift back to $S^{2n+1}$.
    Thus Corollary \ref{Sasaki-ext-sol} is applied through $S$.

    By \cite[Theorem 7.8]{BGS08}, 
    there always exists a Sasaki extremal metric for $\xi_a$.
    Therefore, in view of Corollary \ref{Sasaki-ext-sol},
    there exists a Sasaki Mabuchi soliton if and only if
    the Sasaki Mabuchi constant
    $M_{S^{2n+1},\xi_a}=\sup_{P_a}(l_{a,\rm ext}-n)$
    is less than $1$,
    where $l_{a,\rm ext}$ is the extremal affine function for $\xi_a$.

    Now we determine $l_{a, \rm ext}\in\Aff(\mathcal P_a)$.
    Put $q=(1,\dots,1)\in\mathcal P_a$.
    According to Futaki \cite[Proof of Theorem 1.2]{Fu23},
    the Sasaki Futaki invariant can be written as
    \begin{equation}
        \mathrm{Fut}(l)=-l(q)+\frac{1}{\int_{\mathcal P_a}dx}\int_{\mathcal P_a}l(x)\,dx
    \end{equation}
    for $l\in\Aff({\mathcal P_a})$.
    By definition, $l_{a, \rm ext}$ is given as
    the unique affine function such that
    \begin{equation}
        \mathrm{Fut}(l)=\frac{1}{\int_{\mathcal P_a}dx}
        \int_{\mathcal P_a}l(x)\left( l_{a, \rm ext}(x)-n \right)\, dx
    \end{equation}
    for any $l\in\Aff({\mathcal P_a})$.
    Define $y_i:=a_ix_i$ and put
    \begin{equation}
        l_{a,\rm ext}-n=\sum_{i=0}^nc_iy_i
    \end{equation}
    to determine the coefficients $c_i\in\mathbb R$.
    The defining equation of $l_{a,\rm ext}$ becomes
    \begin{equation}\label{l_ext-y}
        -l(a)+\frac{1}{\int_{\Delta}dy}\int_{\Delta}l\,dy
        =
        \frac{1}{\int_{\Delta}dy}
        \int_{\Delta}l\left( l_{a, \rm ext}-n \right)\, dy
    \end{equation}
    for any $l\in\Aff({\Delta})$,
    where $\Delta:=\Set{y=(y_0,\dots,y_n)\in(\mathbb R_{\geq 0})^{n+1} | \sum_{i=0}^ny_i=1}$
    and $dy=dy_1\cdots dy_n$.
    We use the formulas
    \begin{equation}
        \int_\Delta dy=\frac{1}{n!},
        \quad
        \int_\Delta y_i\, dy=\frac{1}{(n+1)!},
        \quad\text{and}\quad
        \int_\Delta y_iy_j\, dy=\frac{1+\delta_{ij}}{(n+2)!},
    \end{equation}
    where $1\leq i,j\leq n$ and $\delta_{ij}$ is the Kronecker delta.
    Applying $l=1$ to the equation \eqref{l_ext-y}, we obtain
        $\sum_{i=0}^nc_i=0.$
    Applying $l=y_j$ for $j=1,\dots,n$ to the equation \eqref{l_ext-y}, we obtain
    \begin{equation}
        -a_j+\frac{1}{n+1}=\frac{1}{(n+1)(n+2)}
        \left( \sum_{i=0}^nc_i+c_j \right).
    \end{equation}
    Therefore we have the formula
    \begin{eqnarray*}
        l_{a,\rm ext}-n
        &=&(n+2)\sum_{i=0}^n\left( 1-(n+1)a_i \right) y_i \\
        &=&(n+2)\left( 1-(n+1)\sum_{i=0}^na_i^2 x_i \right).
    \end{eqnarray*}

    Since $l_{a,\rm ext}$ is affine, 
    it attains its maximum at a vertex of $\mathcal P_a$.
    The Mabuchi constant can be written as
    
    \begin{equation}
        M_{S^{2n+1},\xi_a}=\sup_{P_a}(l_{a,\rm ext}-n)
        =(n+2)\left( 1-(n+1)\min_ia_i \right).
    \end{equation}
    Consequently, $M_{S^{2n+1},\xi_a}<1$ if and only if
    $\min_ia_i>1/(n+2)$.
    This completes the proof.
\end{proof}

\begin{remark}
    The Reeb field $\xi_a$
    corresponding to $a_0=\cdots=a_n=1/(n+1)$
    is the unique Reeb field admitting a $\widehat{T}$-invariant (toric) Sasaki-Einstein metric.
\end{remark}


\end{document}